\documentclass[11pt]{article}
\usepackage{amssymb,amsmath}
\usepackage{float}
\usepackage{color,fullpage}
\usepackage{cite}
\usepackage[title,titletoc]{appendix}
\usepackage{algorithm,algorithmic}
\UseRawInputEncoding
\usepackage{amsthm}
\newcommand{\RR}{\mathbb{R}}

\newcommand{\dist}{{\mathbf{dist}}}

\usepackage{graphics,graphicx,epstopdf}
\usepackage[tight]{subfigure}

\DeclareMathOperator*{\Min}{minimize\quad}

\newcommand{\st}{\mbox{subject to}}

\newtheorem{theorem}{Theorem}[section]
\newtheorem{lemma}{Lemma}[section]
\newtheorem{definition}{Definition}[section]

\newtheorem{corollary}{Corollary}[section]

\makeatletter

\@addtoreset{equation}{section} \makeatother

\begin{document}

\title{A cut-and-project perspective for linearized Bregman iterations}

\author{Yu-Hong Dai\thanks{LSEC, Academy of Mathematics and Systems Science, Chinese Academy of Sciences, Beijing, China. Email: \texttt{dyh@lsec.cc.ac.cn}}~~~Kangkang Deng\thanks{Department of Mathematics, National University of Defense Technology,
Changsha, Hunan 410073, China.  Email: \texttt{freedeng1208@gmail.com}}~~~ Hui Zhang\thanks{
Department of Mathematics, National University of Defense Technology,
Changsha, Hunan 410073, China.  Email: \texttt{h.zhang1984@163.com}
}
}

\date{\today}

\maketitle

\begin{abstract}
The linearized Bregman iterations (LBreI) and its variants are powerful tools for finding sparse or low-rank solutions to underdetermined linear systems. In this study, we propose a cut-and-project perspective for the linearized Bregman method via a bilevel optimization formulation, along with a new unified algorithmic framework. The new perspective not only encompasses various existing linearized Bregman iteration variants as specific instances, but also allows us to extend the linearized Bregman method to solve more general inverse problems. We provide a completed convergence result of the proposed algorithmic framework, including convergence guarantees to feasible points and optimal solutions, and the sublinear convergence rate.    Moreover, we introduce the Bregman distance growth condition to ensure linear convergence. At last, our findings are illustrated via numerical tests. 

\end{abstract}




\section{Introduction}

The linearized Bregman iterations (LBreI) method, suggested by Darbon and Osher \cite{drabon2007bregman} and introduced in the seminal work \cite{yin2008bregman}, has become a popular method for finding regularized solutions to the underdetermined linear systems $Ax=b$, where $A\in \RR^{m\times n}$ with $m<n$ and $b\in \RR^m$ are given. For example, to seek sparse solutions, it solves the following linearly constrained optimization problem
\begin{equation}\label{BP}
\Min \lambda\|x\|_1+\frac{1}{2}\|x\|^2, ~~\st~~Ax=b,
\end{equation}
where $\lambda>0$ is a regularization parameter. The corresponding algorithmic scheme is straightforward and consists of two steps
\begin{eqnarray*}
x^*_{k+1}&=&x^*_k-A^T(Ax_{k}-b),\\
x_{k+1}&=&\mathcal{S}_\lambda(x_{k+1}^*)
\end{eqnarray*}
initialized with $x_0=x^*_0=0$, where $\mathcal{S}_\lambda(x)=\min\{|x|-\lambda, 0\}\textrm{sign}(x)$ is the component-wise soft shrinkage. At the first step, the dual variable $x_{k}^*$ is updated by performing a gradient descent with $f(x)=\frac{1}{2}\|Ax-b\|^2$ to approach the constraint $Ax=b$; at the second step, the primal variable $x_k$ is updated by applying the soft shrinkage operator to the dual variable $x_k^*$ so that a sparse solution can be obtained. The convergence of the two-step iteration has been analyzed in \cite{cai2009convergence}.


In order to speed up the convergence and extend the LBreI method to solve more general problems than \eqref{BP}, several novel insights into this method have emerged over the past decade. 
The first insight, elucidated in the work by \cite{yin2010analysis}, identifies LBreI as a gradient descent method for the dual problem of \eqref{BP}. This identification enables the application of gradient-based optimization techniques, including line search, Barzilai-Borwein, limited memory BFGS, nonlinear conjugate gradient, and Nesterov's methods, to accelerate the LBreI method.  A second critical insight, proposed by \cite{cai2009linearizedimage} considers the following least square problem
\begin{equation}\label{gBP}
\begin{array}{ccc}
&\Min& ~~\lambda\|x\|_1+\frac{1}{2}\|x\|^2~\\
&\st&~~x\in \arg\min_x \frac{1}{2}\|Ax-b\|^2,
\end{array}
\end{equation}
where the measurement data $b$ can be out the range of $A$. In particular, if $b$ is in the range of $A$, then the inner level optimization reduces to $Az=b$ hence \eqref{gBP} reduces to \eqref{BP}. The authors \cite{lorenz2014linearized} introduce an algorithmic framework using the Bregman projection to address the split feasibility problem as follows:
\begin{equation}\label{CFP}
 \Min~\omega(x), ~\st~ x\in \bigcap_{i=1}^N C_i,
\end{equation}
where $C_i$ are closed convex sets that are given by imposing convex constraints $Q_i\subset \RR^{N_i}$ in the range of a matrix $A_i\in \RR^{N_i\times n}$, i.e., $C_i=\{x\in\RR^n: A_ix\in Q_i\}$.   
Within this framework, the system of linear equations $Ax=b$ can be reformulated as a split feasibility problem, placing the LBreI method as a special case within a broader class of methods that includes popular techniques like the Kaczmarz method \cite{karczmarz1937angenaherte} and the Landweber method \cite{hanke1991accelerated}. Moreover, \eqref{CFP} also covers more generalized problems, including other objective functions $\omega$ and incorporation of non-Gaussian noise models, and box constraints. 

Despite these advancements, it is important to note the primary theoretical limitation of this framework\cite{lorenz2014linearized}: the convergence analysis remains incomplete. In particular, they only give a monotonic decrease in terms of Bregman distance. The convergence to optimal solution and convergence rate are unknown. This naturally gives rise to the following question:
\begin{quote}
{\textit{Can we design a unified algorithmic framework that covers LBreI with a completed convergence result?}}
\end{quote}

We affirmatively answer this question by offering a novel perspective on LBreI through the lens of bilevel optimization. We formulate the above problems as the bilevel problem and propose a cut-and-project method to encompass LBreI and the algorithm outlined in \cite{lorenz2014linearized}, providing a comprehensive convergence result. In particular, we employ the cutting plane and Bregman projection methods to create a straightforward yet broadly applicable algorithmic framework, accompanied by a comprehensive convergence analysis. It is noteworthy that our approach not only provides a clear geometry intuition but also demonstrates significant potency in surpassing existing theoretical frameworks.

Besides the previously mentioned several understandings of LBreI, there are other important extensions. For example, the authors of \cite{cai2010singular} proposed the singular value thresholding (SVT) algorithm which is a matrix-form variant of LBreI and has become a classic method for matrix completion. Additionally, the authors of \cite{benning2017choose} extended the LBreI method to a larger class of non-convex functions for a wider range of imaging applications. 
In this study, we will rediscover the main result in \cite{cai2009linearized} as a special case within our proposed framework, utilizing an essentially different line of thought.

The main contributions of this paper are summarized as follows:

\begin{itemize}
\item
By introducing a bilevel optimization formulation, we present a cut-and-project perspective to reexamine Bregman regularized iterations. This method not only encompasses existing Bregman regularized iterations but also applies to solving more general inverse problems, including certain sparse noise models.

\item 
A detailed analysis is conducted on the impact of step-sizes on the algorithm, particularly identifying the precise range of step-sizes that can guide the selection of practical step-sizes. Furthermore, we provide a convergence guarantee to a feasible point with a large range of step-sizes, which are wider compared to existing literature.

\item
To establish the convergence of the proposed algorithmic framework to the optimal solution of the problem, we provide a unified convergence condition. When applied to specific examples, our condition is less restrictive compared to conditions proposed in \cite{zhang2023revisiting}.
\item
A Bregman distance growth condition is introduced to ensure the linear convergence of the algorithm.   Additionally, we derive sublinear convergence results for the algorithm, which has only been analyzed for specific problems so far, such as the SVT algorithm \cite{cai2010singular} or cases where the regularization function is smooth.

\end{itemize}

The outline of the paper is as follows. Section \ref{se2} presents some preliminaries and discussions. 
In Section \ref{sec:alg}, we give the bilevel optimization formulation, present the proposed algorithm framework and give the descent lemma. The main convergence results, including the convergence to a feasible point, the convergence to the optimal solution, the linear convergence and sublinear convergence, are provided in Section \ref{se5}. The numerical
experiments of our algorithm are given in Section \ref{sec6}.

\section{Preliminaries}\label{se2}

\subsection{Notation}
In this paper, we restrict our analysis to real finite dimensional spaces $\RR^n$. We use $\langle \cdot, \cdot\rangle$ to denote the inner product and  $\|\cdot\|$ to denote the Euclidean norm. For a multi-variables function $f(x,y)$, we use $\nabla_xf$ (respectively, $\nabla_yf$) to denote the gradient of $f$ with respective to $x$ (respectively, $y$). Let $Q\subset \RR^n$ be a nonempty set; the distance function onto the set $Q$ is defined by   
  $\dist(x,Q):=\inf_{y\in Q}\|x-y\|$. The indicator function of a set $\mathcal{C}$, denoted by $\delta_{\mathcal{C}}$, is set to be zero on $\mathcal{C}$ and $+\infty$ otherwise.

\subsection{Convex analysis tools}
Some basic notations and facts about convex analysis will be used in our results.
\begin{definition}\label{sc0}
A function $\omega:\RR^d\rightarrow \RR$ is convex if for any $\alpha\in [0,1]$ and $u, v\in \RR^d$, we have
\begin{equation*}
\omega(\alpha u+(1-\alpha)v)\leq \alpha \omega(u)+(1-\alpha)\omega(v);
\end{equation*}
and strongly convex with modulus $\mu> 0$ if for any $\alpha\in [0,1]$ and $u, v\in \RR^d$, we have
\begin{equation*}
\omega(\alpha u+(1-\alpha)v)\leq \alpha \omega(u)+(1-\alpha)\omega(v)-\frac{1}{2}\mu\alpha(1-\alpha)\|u-v\|^2. \label{SC1}
\end{equation*}
Furthermore, $\omega$ is strictly convex if for any $u\neq v$ and $\alpha\in (0,1)$ we have 
\begin{equation*}
\omega(\alpha u+(1-\alpha)v)< \alpha \omega(u)+(1-\alpha)\omega(v).
\end{equation*}
\end{definition}

\begin{definition}
Let $\omega:\RR^d\rightarrow \RR$ be a convex function. The subdifferential of $\omega$  at $u\in \RR^d$ is defined as
$$\partial \omega (u):= \{ u^* \in \RR^d: \omega(v)\geq \omega(u)+ \langle u^*, v-u\rangle,\quad \forall v\in \RR^d \}.$$
The elements of $\partial \omega(u)$ are called the subgradients of $\omega$ at $u$.
\end{definition}

The subdifferential generalizes the classical concept of differential because of the well-known fact that $\partial \omega(u)=\{\nabla \omega(u)\}$ when the function $\omega$ is differentiable.
In terms of the subdifferential, the strong convexity in Definition \ref{sc0} can be equivalently stated as \cite{hiriart2004fundamentals}: For any $u, v\in \RR^d$ and $v^*\in \partial\omega(v)$, we have
\begin{equation}\label{sc01}
 \omega(u)\geq \omega(v)+\langle v^*, u-v\rangle +\frac{\mu}{2}\|u-v\|^2.
\end{equation}

\begin{definition}
Let $\omega:\RR^d\rightarrow \RR$ be a convex function. The conjugate of $\omega$ is defined as
$$\omega^*(u^*)=\sup_{v\in \RR^d}\{\langle u^*,v \rangle -\omega(v)\}.$$
\end{definition}

\begin{definition}
A function $\omega:\RR^d\rightarrow \RR$ is gradient-Lipschitz-continuous with modulus $L>0$ if it is continuously differentiable over $\RR^n$ and its gradient is Lipschitz continuous in the sense that for any $u, v\in \RR^d$, we have
\begin{equation*}
\|\nabla \omega(u)-\nabla \omega(v)\|\leq L\|u-v\|. \label{Lip}
\end{equation*}
\end{definition}

The following facts are well-known, and can be found in the classic textbooks \cite{rockafellar1997convex} and \cite{hiriart2004fundamentals}.

\begin{lemma}\label{scLip}
Let $\omega:\RR^d\rightarrow \RR$ be a strongly convex function with modulus $\mu> 0$. Then we have that
 \begin{itemize}
   \item its conjugate $\omega^*$ is gradient-Lipschitz-continuous with modulus $\frac{1}{\mu}$;
   \item the conditions $\omega(u)+\omega^*(u^*)=\langle u, u^*\rangle$, $u^*\in \partial \omega(u)$, and $u\in\partial \omega^*(u^*)$ are equivalent.
 \end{itemize}

\end{lemma}

\subsection{Bregman distance tools}
The Bregman distance, originally introduced in \cite{bregman1967relaxation}, is a very powerful concept in many fields where distances are involved. Recently, many variants of Bregman distances have appeared, as seen in  \cite{bauschke1997legendre,kiwiel1997free,reem2019re}. For simplicity as well as generality, we choose the Bregman distance defined by a strongly convex function.
\begin{definition}
Let $\omega:\RR^d\rightarrow \RR$ be a strongly convex function with modulus $\mu> 0$. The Bregman distance $D_\omega^{v^*}(u,v)$ between $u, v\in \RR^d$ with respect to $\omega$ and a subgradient $v^*\in\partial \omega(v)$ is defined by
\begin{equation}\label{Breg}
D_\omega^{v^*}(u,v):=\omega(u)-\omega(v)-\langle v^*, u-v\rangle.
\end{equation}
\end{definition}
In the following, we state three basic facts about the Bregman distance, which will be used later in our analysis. It should be pointed out that the results are well-known (see e.g. \cite{kiwiel1997free,kiwiel1997proximal}). We list them here, along with a brief proof, for completeness.
\begin{lemma}\label{lemBreg}
Let $\omega:\RR^d\rightarrow \RR$ is strongly convex with modulus $\mu> 0$. For any $u, p, q\in \RR^d$ and $p^*\in \partial \omega (p), q^*\in \partial \omega (q)$, we have that
\begin{equation}\label{Bregdis1}
D_\omega^{p^*}(u,p) - D_\omega^{q^*}(u,q) + D_\omega^{q^*}(p,q)=\langle q^*-p^*, u-p\rangle,
\end{equation}
\begin{equation}\label{Bregdis2}
D_\omega^{q^*}(p,q)=D_{{\omega}^*}^p(q^*,p^*),
\end{equation}
and
\begin{equation}\label{Bregdis3}
D_\omega^{q^*}(p,q)\geq \frac{\mu}{2}\|p-q\|^2.
\end{equation}
\end{lemma}

\begin{proof}
The equality \eqref{Bregdis1} follows from \eqref{Breg}, and the inequality \eqref{Bregdis3} from \eqref{sc01} and \eqref{Breg}. In order to obtain \eqref{Bregdis2}, we derive that
\begin{eqnarray}
\begin{array}{lll}
D_\omega^{q^*}(p,q) &= & \omega(p)-\omega(q)-\langle q^*, p-q\rangle    \\
  &= &\langle p, p^*\rangle -\omega^*(p^*) -\langle q, q^*\rangle +\omega^*(q^*) -\langle q^*, p-q\rangle  \\
&= &\omega^*(q^*)-\omega^*(p^*) -\langle p, q^*-p^*\rangle\\
&=& D_{{\omega}^*}^p(q^*,p^*),
\end{array}
\end{eqnarray}
where the second and fourth lines follow by using the second part of Lemma \ref{scLip} and the condition $p^*\in \partial \omega (p), q^*\in \partial \omega (q)$.
\end{proof}

Now, we introduce the concept of Bregman projection.
\begin{definition}[Bregman projection \cite{kiwiel1997proximal}]\label{DefofBreg}
Let $\omega:\RR^d\rightarrow \RR$ be a strongly convex function with modulus $\mu> 0$, $C\subset \RR^n$ be a nonempty closed convex set. Given $x\in \RR^n$ and $x^*\in \partial \omega(x)$, the Bregman projection of $x$ onto $C$ with respect to $\omega$ and $x^*$ is the point $\Pi_C^{x^*}(x)\in C$ such that
\begin{equation}\label{BregP}
D^{x^*}_\omega(\Pi_C^{x^*}(x), x)=\min_{y\in C} D^{x^*}_\omega(y,x).
\end{equation}
\end{definition}
Note that the point $\Pi_C^{x^*}(x)$ exists and is unique due to the strong convexity of the objective function in \eqref{BregP}; besides, the notation $\Pi_C^{x^*}(x)$ is dependent on the function $\omega$ and reduces to the orthogonal projection, denoted by $\mathcal{P}_C(x)$, when $\omega(x)=\frac{1}{2}\|x\|^2.$ Similar to the traditional orthogonal projection theorem, the Bregman projection can be characterized by a variational inequality.
\begin{lemma}[Generalized projection theorem\cite{kiwiel1997proximal}]\label{gptlemm}
Let $\omega:\RR^d\rightarrow \RR$ be a strongly convex function with modulus $\mu> 0$ and $C\subset \RR^n$ be a nonempty closed convex. Then a point $z\in C$ is the Bregman projection of $x$ onto $C$ with respect to $\omega$ and $x^*\in \partial \omega (x)$ if and only if there is some point $z^*\in \partial \omega(z)$ such that one of the following equivalent conditions are fulfilled:
\begin{equation}\label{GPT1}
\langle z^*-x^*, y-z\rangle \geq 0, ~~\forall ~~y\in C,
\end{equation}
\begin{equation}\label{GPT2}
D^{z^*}_\omega(y,z)\leq D^{x^*}_\omega(y,x)-D^{x^*}_\omega(z,x), ~~\forall ~~y\in C.
\end{equation}
Any such $z^*$ is called an admissible subgradient for $z=\Pi_C^{x^*}(x).$ In particular, when $\omega(x)=\frac{1}{2}\|x\|^2$, the condition $x^*\in \partial \omega (x)$  reduces to $x^*=x$ and hence the inequality \eqref{GPT1} just becomes the orthogonal projection characterization.
\end{lemma}

\section{The proposed algorithmic framework}\label{sec:alg}
\subsection{The problem formulation}\label{se3}
We are interested in the following bilevel optimization problem which consists of inner and outer levels. The inner level is the classical optimization problem
\begin{equation}\label{innopt}
\Min_{x\in \RR^n} f(x),
\end{equation}
where the optimal set is denoted by $\overline{X}$, assumed to be nonempty. The out level seeks to find the optimal solution of problem \eqref{innopt} which minimizes a given strongly convex function $\omega$:
\begin{equation}\label{outopt}
\Min_{x\in \overline{X}} \omega(x).
\end{equation}
The inner level objective function $f$ and the outer level objective function $\omega$ are assumed to satisfy the following:
\begin{itemize}
  \item $\omega$ is strongly convex over $\RR^n$ with parameter $\mu>0$;
  \item $f$ is convex and gradient-Lipschitz-continuous with modulus $L>0$.
\end{itemize}

By the strong convexity of $\omega$, problem \eqref{outopt} has a unique solution, denoted by $\bar{x}$. Putting the ingredients above together, we obtain the following bilevel optimization:
\begin{equation}\label{BilevelOpt}
\begin{aligned}
&\Min ~~ \omega(x)~\\
&\st~~x\in \arg\min_x f(x).
\end{aligned}
\end{equation}
The terminology of bilevel optimization, mainly following from papers \cite{beck2014first}, is simpler than classical bilevel optimization where the solution to the inner level optimization usually depends on variables or parameters.

This bilevel optimization encompasses various specific problems. When $\omega(x) = \lambda\|x\|_1 + \frac{1}{2}\|x\|^2$ and $f(x) = \frac{1}{2}\|Ax-b\|^2$, \eqref{BilevelOpt} simplifies to the least squares problem \eqref{gBP}. Furthermore, if $b\in \mathcal{R}(A)$, the inner-level optimization reduces to $Ax=b$, and thus \eqref{BilevelOpt} covers \eqref{BP}. Moreover, \eqref{BilevelOpt} encompasses the split feasibility problem \eqref{CFP} by defining $f(x):=\sum_{i=1}^N \dist^2(A_ix,Q_i)$, which is convex and gradient-Lipschitz-continuous. This is due to that minimizing $f$ is equivalent to finding $x\in \bigcap_{i=1}^N C_i$.

\subsection{Sketch of idea}
Our proposed iterative scheme will be grounded on a straightforward and intuitive geometric understanding. Let us start with a given point $x$ which does not lie in the optimal set $\overline{X}$. We can then construct a hyperplane to separate this given point $x$ from the optimal set $\overline{X}$. In order to gradually approach the optimal set $\overline{X}$, we project the given point $x$ onto the hyperplane, thereby obtaining updated points that progressively draw nearer to the optimal set. However, our aim is not solely for the iterates to converge to any solution of the inner-level optimization but specifically to the unique solution $\bar{x}$ of the bilevel optimization problem. To achieve this, we select the objective function $\omega$ of the outer-level optimization as a kernel to generate the Bregman distance $D_\omega$. Subsequently, we utilize the Bregman projection to replace the orthogonal projection. 

From the discussion above, the main idea of our method consists of two steps: constructing a halfspace (with 
 a hyperplane boundary) and performing the Bregman projection onto the halfspace. The first step relies on the Baillon-Haddad theorem \cite{baillon1977quelques}, which states that the convexity and Lipschitz continuity of gradient of $f$ is equivalent to the following inequality
$$\langle \nabla f(u)-\nabla f(v), u-v\rangle \geq \frac{1}{L}\|\nabla f(u)-\nabla f(v)\|^2.$$
Note that for any $x\in \overline{X}$, $\nabla f(x)=0$. Then, for any given point $z$, we have the following inclusion 
$$\overline{X}\subset H_z:= \{x\in\RR^n: \langle \nabla f(z), z-x\rangle \geq \frac{1}{L}\|\nabla f(z)\|^2\},$$
where $H_z$ is a halfspace which separates the given point $z$ and the optimal set $\overline{X}$. In particular, if $\nabla f(z)\neq 0$, then the separation is strict in the sense that $z$ is not in $H_z$ since in that case $z$ fails to satisfy the inequality $\langle \nabla f(z), z-x\rangle \geq \frac{1}{L}\|\nabla f(z)\|^2$. The next step is to project $z$ onto the halfspace $H_z$. Denote the project operator onto the halfspace $H_z$ as $\mathcal{P}$; then for any point $\hat{x}\in \overline{X}$, we can derive that 
$$\|\mathcal{P}(z)-\hat{x}\|=\|\mathcal{P}(z)-\mathcal{P}(\hat{x})\|\leq \|z-\hat{x}\|,$$ 
where the first relationship is due to $\hat{x}\in \overline{X}\subset H_z$ and the inequality follows from the nonexpansiveness of the project operator. Therefore, it is possible to get closer to the optimal set $\overline{X}$ after the project operation. 

The idea of our method can be simply summarized as follows: 
\begin{description}
  \item[s1.] Construct cutting halfspaces: $H_k=\{x\in\RR^n: \langle \nabla f(x_k), x_k-x\rangle \geq \frac{1}{L}\|\nabla f(x_k)\|^2\}.$
  \item[s2.] Update via Bregman projections: $x_{k+1}:=\Pi_{H_k}^{x^*_k}(x_k)$.
\end{description}

Lastly, we would like to mention two closely related works \cite{lorenz2014linearized} and \cite{beck2017first}, which have greatly inspired us. In \cite{lorenz2014linearized}, the authors presented a generalized algorithmic framework based on Bregman projections, which includes LBreI as a special case. However, their framework is restricted solely to solving split feasibility problems \eqref{CFP}. On the other hand, in \cite{beck2017first}, the authors introduced a novel algorithm based on the cutting plane method to address a class of bilevel convex problems. Nevertheless, their method necessitates the outer-level function to be smooth, and at each iteration, they are required to solve a constrained optimization problem, which may pose computational challenges. Our approach is motivated by amalgamating the strengths of these two works, aiming to overcome their respective limitations.

\subsection{The Bregman projection as a primal-dual problem}
The key computation in our method involves performing Bregman projections. In this part, we will formulate it as a primal-dual problem. To this end, we assume that $\nabla f(x_k)\neq 0$ and let $a_k:=\nabla f(x_k)$ and $$\beta_k:=\langle \nabla f(x_k), x_k\rangle-\frac{1}{L}\|\nabla f(x_k)\|^2.$$ Then the halfspaces defined in the first step of our method can be simplified as follows
    \begin{equation}\label{halfsp}
    H_k=\{x\in\RR^n: \langle a_k, x\rangle\leq \beta_k\}.
    \end{equation}
 From Definition \ref{DefofBreg},  the Bregman projections have the following formulation
 $$\Pi_{H_k}^{x^*_k}(x_k)=\arg\min_{x\in H_k} D_\omega^{x_k^*}(x,x_k),$$
 which is equivalent to solving the following constrained optimization
 \begin{eqnarray}\label{P0}
&\Min&  \omega(x)-\langle x_k^*, x\rangle \\ \nonumber 
&\st& \langle a_k, x\rangle \leq \beta_k. 
\end{eqnarray}
In order to rewrite \eqref{P0} as an unconstrained optimization, we let $\Omega_k=\{s\in\RR: s\leq \beta_k\}$, $\varphi(s)=\delta_{\Omega_k}(s)$, and $\psi(x)=\omega(x)-\langle x_k^*, x\rangle$; then  \eqref{P0} can be equivalently rewritten as 
    \begin{equation}\label{P}
    \Min_{x\in\RR^n} \psi(x)+\varphi(a_k^Tx),
    \end{equation}
 whose Fenchel-Rockafellar dual problem is 
   \begin{equation}\label{D}
    \Min_{t\in\RR} \psi^*(-ta_k)+\varphi^*(t).
    \end{equation}
 Note that $\psi^*(x)=\omega^*(x_k^*+x)$ and 
\begin{eqnarray*}
\varphi^*(t)=\delta^*_{\Omega_k} (t):=
\left\{\begin{array}{lll}
 \beta_kt,       &\textrm{if} ~~t\geq 0, \\
 +\infty, &\textrm{otherwise}.
\end{array} \right.
\end{eqnarray*} 
 The dual problem \eqref{D} can be further written as 
    \begin{equation}\label{D0}
    \Min_{t\geq 0} \omega^*(x_k^*-t a_k)+\beta_k t.
    \end{equation}
 Let $x\in\RR^n, t\in\RR$. Then $x$ is a primal solution and $t$ is a dual solution if and only if the KKT condition holds:
 \begin{eqnarray}
\left\{\begin{array}{rll}
 -t a_k&\in &\partial \psi(x)=\partial\omega(x)-x_k^*, \\
t&\in& \partial \varphi(a_k^Tx)=N_{\Omega_k}(a_k^\top x),
\end{array} \right.
\end{eqnarray}
where the normal cone
 \begin{eqnarray*}
N_{\Omega_k}(a_k^\top x)= \left\{\begin{array}{ll}
 \RR_+,& a_k^Tx=\beta_k, \\
\{0\},& a_k^Tx<\beta_k.
\end{array} \right.
\end{eqnarray*}
We observe that the case of $t\in \{0\}$ and $a_k^Tx<\beta_k$ can be excluded.  Indeed, let us focus on the dual objective function 
$$g(t)=\omega^*(x_k^*-t a_k)+\beta_k t.$$
Its derivative can be calculated as 
 $g^\prime(t)=-a_k^T\nabla \omega^*(x_k^*-t a_k)+\beta_k$ 
and hence we can derive that
\begin{eqnarray*}
\begin{array}{lll}
g^\prime(0) &= &  -a_k^T\nabla \omega^*(x_k^*)+\beta_k\\
  &= & -a_k^Tx_k+\beta_k\\
&= &-\frac{1}{L}\|\nabla f(x_k)\|^2<0,
\end{array}
\end{eqnarray*}
where we have utilized the fact of $x_k^*\in \partial \omega(x_k)$, the expressions of $a_k$ and $\beta_k$, and the assumption $\nabla f(x_k)\neq 0$. Hence, the function $g(t)$ must be decreasing for small enough $t>0$. In other words, the dual solution $t$ must be obtained as a positive number. Therefore, the dual problem can be equivalently rewritten as 
    \begin{equation}\label{D00}
    \Min_{t> 0} \omega^*(x_k^*- t a_k)+\beta_k t,
    \end{equation}
    and the KKT condition can be reformulated as the following form after some simple transformations:
  \begin{eqnarray}\label{KKT}
  \left\{\begin{array}{rll}
x &= &\nabla \omega^*(x_k^*-t a_k), \\
a_k^Tx&=& \beta_k,\\
t&>& 0.
\end{array} \right.
\end{eqnarray}  
In particular, the condition $a_k^Tx= \beta_k$ implies that the Bregman projection must lie in the hyperplane $\{x\in\RR^n: \langle \nabla f(x_k), x_k-x\rangle = \frac{1}{L}\|\nabla f(x_k)\|^2\}.$ Based on the KKT condition \eqref{KKT}, we can deduce the following result.

\begin{theorem}\label{theo:projection}
  Let $\hat{t}_k$ be a dual solution to \eqref{D00}. Then the Bregman projection of $x_k$ onto the halfspace $H_k$ can be obtained by the following formulation
  \begin{equation}\label{Bproj}
    \Pi_{H_k}^{x^*_k}(x_k)=\nabla \omega^*(x_k^*-\hat{t}_k a_k)
    \end{equation}
    and moreover $\hat{t}_k$ belongs to $(0,\frac{\mu}{L}]$. On the other hand, if the point $\nabla\omega^*(x_k^*-a_k\tilde{t}_k)$ with $\tilde{t}_k\in (0,\frac{\mu}{L}]$ lies in the hyperplane $\{x\in\RR^n: \langle \nabla f(x_k), x_k-x\rangle = \frac{1}{L}\|\nabla f(x_k)\|^2\}$, then it must be the unique point of the Bregman projection, that is 
     \begin{equation}\label{Bproj0}
  \nabla\omega^*(x_k^*-a_k\tilde{t}_k)=\nabla\omega^*(x_k^*-\hat{t}_k a_k)= \Pi_{H_k}^{x^*_k}(x_k).
    \end{equation}
\end{theorem}

\begin{proof}
  Since $\hat{t}_k$ is an exact solution to \eqref{D00}, by the first equation of the KKT condition \eqref{KKT} we deduce that $\nabla \omega^*(x_k^*-\hat{t}_k a_k)$ must be the unique primal solution to \eqref{P} and hence the desired result \eqref{Bproj} holds. In terms of the second equation of \eqref{KKT}, we have 
  $$a_k^T\nabla \omega^*(x_k^*-\hat{t}_k a_k)=\beta_k.$$
  Using the expressions of $a_k$ and $\beta_k$ and the fact that $x_k=\nabla \omega^*(x_k^*)$, we further get 
  \begin{equation}\label{eqinn1}
    \langle \nabla f(x_k), \nabla \omega^*(x_k^*)-\nabla \omega^*(x_k^*-\hat{t}_k\cdot\nabla f(x_k))\rangle=\frac{1}{L}\|\nabla f(x_k)\|^2.
  \end{equation}
  By the Lipschitz continuity of $\nabla \omega^*$, we can bound the left-hand side of \eqref{eqinn1} as follows:
    \begin{equation}\label{eqinn2}
    \langle \hat{t}_k\cdot\nabla f(x_k), \nabla \omega^*(x_k^*)-\nabla \omega^*(x_k^*-\hat{t}_k\cdot\nabla f(x_k))\rangle\geq \frac{1}{\mu}\|\hat{t}_k\cdot\nabla f(x_k)\|^2.
  \end{equation}
  Combining \eqref{eqinn1} and \eqref{eqinn2}, we get 
  $$\frac{\hat{t}_k}{L}\|\nabla f(x_k)\|^2\geq \frac{1}{\mu}\|\hat{t}_k\cdot\nabla f(x_k)\|^2,$$
  which leads to $0<\hat{t}_k\leq \frac{\mu}{L}$. Therefore, the proof of the first part is completed. 
  
 In order to show the second part, we introduce two auxiliary points:
   \begin{eqnarray*} 
\tilde{x}_{k+1} &= &\nabla \omega^*(x_k^*-a_k\tilde{t}_k), \\
\tilde{x}_{k+1}^*&=& x_k^*-a_k\tilde{t}_k. 
\end{eqnarray*} 
Then, $\tilde{x}_{k+1}^*\in \partial\omega(\tilde{x}_{k+1})$ and the point $\tilde{x}_{k+1}$ belongs to $H_k$ due to the assumption that $\tilde{x}_{k+1}$ lies in the hyperplane $\{x\in\RR^n: \langle \nabla f(x_k), x_k-x\rangle = \frac{1}{L}\|\nabla f(x_k)\|^2\}$.  By the uniqueness of the Bregman projection, it suffices to show that $\tilde{x}_{k+1}$ is the Bregman projection of $x_k$ onto the halfspace $H_k$. Based on Lemma \ref{gptlemm}, we only need to verify the following relationship: 
\begin{equation*}
\langle \tilde{x}_{k+1}^*-x_k^*, x-\tilde{x}_{k+1}\rangle \geq 0, ~~\forall x\in H_k.
\end{equation*}
Actually, we can derive that for any $x\in H_k$,
   \begin{eqnarray*} 
& &\langle \tilde{x}_{k+1}^*-x_k^*, x-\tilde{x}_{k+1}\rangle\\
&=&\langle -\tilde{t}_k\nabla f(x_k), x-\nabla \omega^*(x_k^*-a_k\tilde{t}_k)\rangle\\
&=&\tilde{t}_k(\langle  \nabla f(x_k),  \nabla \omega^*(x_k^*-a_k\tilde{t}_k)\rangle- \langle \nabla f(x_k), x \rangle)\\
&\geq& \tilde{t}_k(\langle  \nabla f(x_k),  \nabla \omega^*(x_k^*-a_k\tilde{t}_k)\rangle- \beta_k)\\
&=& -\tilde{t}_k\left(\langle  \nabla f(x_k),  \nabla\omega^*(x_k^*)-\nabla \omega^*(x_k^*-a_k\tilde{t}_k)\rangle- \frac{1}{L}\|\nabla f(x_k)\|^2\right)\\
&\geq& -\tilde{t}_k\left(\|\nabla f(x_k)\|\cdot\|\nabla\omega^*(x_k^*)-\nabla \omega^*(x_k^*-a_k\tilde{t}_k)\|- \frac{1}{L}\|\nabla f(x_k)\|^2\right)\\
 &\geq& -\tilde{t}_k\left(\frac{\tilde{t}_k}{\mu}\|\nabla f(x_k)\|^2- \frac{1}{L}\|\nabla f(x_k)\|^2\right)\geq0, 
\end{eqnarray*} 
where the first inequality follows from the condition $x\in H_k$, the second from the Cauchy-Schwartz inequality, the third from the Lipschitz continuity of $\nabla \omega^*$, and the last one is due to the condition $\tilde{t}_k\in (0,\frac{\mu}{L}]$. Thus, from the KKT condition \eqref{KKT}, we can see that $\tilde{t}_k$ must be a solution to the dual problem. Combining with the deduced result that any exact solutions must belong to $(0,\frac{\mu}{L}]$, we can further conclude that $(0,\frac{\mu}{L}]$ is just the set of dual solutions. This completes the proof. 
\end{proof}

\subsection{Iterate scheme and descent lemma}
Based on the previous study, we are ready to present our method concretely. The detailed algorithm is provided in Algorithm \ref{ITS}, where the main iterative scheme consists of two steps: updating the primal variable $x_{k+1}$ and the dual variable $x_{k+1}^*$. Here, we do not specify the step-size $t_k$, which can be obtained through various choices.  As shown in Theorem \ref{theo:projection}, when the step-size $t_k$ is chosen as an exact solution of the dual problem \eqref{D00}, the update rule of $x_{k+1}$ in Algorithm \ref{ITS} is strictly equivalent to the Bregman projection onto $H_{k}$, i.e.,  $x_{k+1}=\Pi_{H_k}^{x^*_k}(x_k)$.  It should be emphasized that achieving precise solutions $t_k$ necessitates a requirement for computing the Lipschtz constant $L$ of function $f$, which may be relatively costly in certain situations. Nevertheless, we will demonstrate that convergence results can be established across a range of step-size $t_k$. Moreover, we also give the numerical comparison of our algorithm with different step-size $t_k$ in the experiment.

\begin{algorithm}[htb]
   \caption{Bregman projection method based on cutting halfspaces}
   \label{ITS}
\begin{algorithmic}[1]
   \STATE {\bfseries Initialize:  } choose $x_0\in\RR^n$ and $x_0^*\in \partial \omega(x_0)$.
   \FOR{$k = 0,\cdots, $}
   \STATE Update the step-size $t_k$ via a certain rule. 
   \STATE Update the primal variable $x_{k+1}$:
   \begin{equation}
       x_{k+1} = \nabla \omega^*(x_k^*-t_k\cdot\nabla f(x_k)).
   \end{equation}
   \STATE Update the dual variable $x_{k+1}^*$:
   \begin{equation}\label{eq:iter:xkstar}
       x_{k+1}^*= x_k^*-t_k\cdot\nabla f(x_k).
   \end{equation}
   \ENDFOR
\end{algorithmic}
\end{algorithm}

Algorithm \ref{ITS} encompasses a broad range of existing algorithms through appropriate choices of functions $f$ and $\omega$. Specifically, when $f(x) = \frac{1}{2}\|Ax - b\|^2$ and $\omega(x) = \lambda \|x\|_1 + \frac{1}{2}\|x\|^2$, it reduces to the linearized Bregman iteration method, as evidenced by Corollary \ref{coro001}. Moreover, setting $f(x) = \frac{1}{2}\dist^2(Ax, Q)$, where $Q$ is a given closed convex set in $\mathbb{R}^m$, yields the iterative method proposed in \cite{lorenz2014linearized} (referenced as \eqref{Lbi01}). This illustrates the versatility of Algorithm \ref{ITS} in encompassing various optimization methods tailored to specific problem structures.

 The following lemma shows that Algorithm \ref{ITS} achieves a monotonic decrease in terms of  Bregman distance with different choices of the step-size $t_k$.  
 \begin{lemma}[Basic descent lemma]\label{lembasic}
Let $(x_k,x_k^*)$ be the $k$-iterate generated in Algorithm \ref{ITS}.  If the step-size $t_k$ is an exact solution to the dual problem \eqref{D00} or equals to $\frac{\mu}{L}$ exactly, then we have 
\begin{equation}\label{desc1}
 D_\omega^{x_{k+1}^*}(x,x_{k+1})\leq  D_\omega^{x_{k}^*}(x,x_{k})-\frac{\mu}{2L^2}\|\nabla f(x_k)\|^2, ~\forall x\in H_k.
\end{equation}
If the step-size $t_k$ lies in the interval $(0,\frac{2\mu}{L}]$, then we have 
\begin{equation} \label{desc2}
 D_\omega^{x_{k+1}^*}(x,x_{k+1})\leq  D_\omega^{x_{k}^*}(x,x_{k})-\left(\frac{t_k}{L}-\frac{t_k^2}{2\mu}\right)\|\nabla f(x_k)\|^2, ~\forall x\in H_k.
\end{equation}
\end{lemma}
 \begin{proof}
 First of all, let us show \eqref{desc1}, where the iterate sequences $\{x_k\}$ and $\{x_k^*\}$ are generated by Algorithm \ref{ITS} with $t_k=\hat{t}_k$ being the exact solution to the dual problem \eqref{D00}. From Lemma \ref{scLip} and the definition of Bregman distance, we can express the Bregman distance in the following form
   $$D_\omega^{x^*}(y,x)=\omega^*(x^*)-\langle x^*, y\rangle +\omega(y), ~\forall x^*\in \partial \omega (x).$$
   Using this expression and the updated formulation \eqref{eq:iter:xkstar} of $x_{k+1}^*$, we derive that 
      \begin{eqnarray*} 
D_\omega^{x_{k+1}^*}(x,x_{k+1})&=& \omega^*(x_{k+1}^*)-\langle x_{k+1}^*, x\rangle +\omega(x)\\
&=&\omega^*(x_k^*-t_k\cdot\nabla f(x_k))-\langle x_k^*, x\rangle +\langle t_k\cdot\nabla f(x_k), x\rangle +\omega(x).
\end{eqnarray*} 
    By the definition of $H_k$, we know that if $x\in H_k$, then 
    $\langle \nabla f(x_k),x\rangle \leq \beta_k.$
    Thus, combining with the Lipschitz continuity of $\nabla \omega^*$ and the fact that $t_k$ is the minimizer to the dual objective function $\omega^*(x_k^*-t a_k)+\beta_k t$, we can further derive that for any $x\in H_k$ and $t>0$,
\begin{eqnarray*} 
&&D_\omega^{x_{k+1}^*}(x,x_{k+1})\\
&\leq & \omega^*(x_k^*-t_k\cdot\nabla f(x_k)) +t_k\cdot\beta_k -\langle x_k^*, x\rangle+\omega(x)\\
&\leq & \omega^*(x_k^*-t\cdot\nabla f(x_k)) +t\cdot\beta_k-\langle x_k^*, x\rangle +\omega(x)\\
&\leq & \omega^*(x_k^*)-\langle \nabla\omega^*(x_k^*), t\nabla f(x_k)\rangle +\frac{t^2}{2\mu}\|\nabla f(x_k)\|^2 +t\cdot\beta_k-\langle x_k^*, x\rangle +\omega(x)\\
&= & \omega^*(x_k^*)-t\langle x_k, \nabla f(x_k)\rangle +\frac{t^2}{2\mu}\|\nabla f(x_k)\|^2 +t\cdot\beta_k-\langle x_k^*, x\rangle +\omega(x)\\
&= & \omega^*(x_k^*)-\langle x_k^*, x\rangle +\omega(x)+\left(\frac{t^2}{2\mu}-\frac{t}{L}\right)\|\nabla f(x_k)\|^2 \\
&= & D_\omega^{x_{k}^*}(x,x_{k})+\left(\frac{t^2}{2\mu}-\frac{t}{L}\right)\|\nabla f(x_k)\|^2, 
\end{eqnarray*} 
where the expression of $\beta_k$ and the fact that $x_k=\nabla\omega^*(x_k^*)$ have been used. Therefore, for any $x\in H_k$ we have that 
\begin{eqnarray*} 
D_\omega^{x_{k+1}^*}(x,x_{k+1})
&\leq & D_\omega^{x_{k}^*}(x,x_{k})+\inf_{t>0}\left(\frac{t^2}{2\mu}-\frac{t}{L}\right)\|\nabla f(x_k)\|^2\\
&=&D_\omega^{x_{k}^*}(x,x_{k})-\frac{\mu}{2L^2}\|\nabla f(x_k)\|^2.
\end{eqnarray*}
Note that the descent property can also be obtained with $t_k=\frac{\mu}{L}$. Hence, the proof of the first part has been done. The second part follows by observing that when $t\in (0,\frac{2\mu}{L}]$, the term $\left(\frac{t}{L}-\frac{t^2}{2\mu}\right)$ is nonnegative so that \eqref{desc2} is always a descent property. Thus, we complete the whole proof. 
 \end{proof}
 
\section{Convergence Analysis}\label{se5}
\subsection{Convergence to a feasible point}
\begin{lemma}\label{lemmon}
Let $(x_k,x_k^*)$ be the $k$-iterate generated in Algorithm \ref{ITS}. 
 Assume that the sequence $\{x_k\}$ converges to a point $x$. If for a given point $y$ and a given subgradient sequence $\{x_k^*\}$ satisfying $x_k^*\in \partial \omega(x_k)$, the sequence $\{D_\omega^{x_k^*}(y,x_k)\}$ converges, then there must exist a point $x^*$ in $\partial \omega(x)$ such that
\begin{equation}
\lim_{k\rightarrow \infty}D_\omega^{x_k^*}(y,x_k)=D_\omega^{x^*}(y,x).
\end{equation}
\end{lemma}

\begin{proof}
First of all, the convergence of $\{x_k\}$ implies its boundedness, which further implies the boundedness of the subgradient sequence $\{x_k^*\}$ by invoking Theorem \cite[Theorem 3.16]{beck2017first}. Hence, there exists a convergent subsequence $\{x_{k_l}^*\}$ whose limit point is denoted by $x^*$. Since $\{x_k\}$ converges to the point $x$, its any subsequence, including the subsequence $\{x_{k_l}\}$, must converge to $x$ as well. The function $\omega$ must be continuous by noting that it is a real-valued convex function. Based on these facts, we derive that
\begin{eqnarray*}
\lim_{k\rightarrow \infty}D_\omega^{x_k^*}(y,x_k) &= & \lim_{l\rightarrow \infty}D_\omega^{x_{k_l}^*}(y,x_{k_l})   \\
  &= &\lim_{l\rightarrow \infty}(\omega(y)-\omega(x_{k_l})-\langle x_{k_l}^*, y-x_{k_l}\rangle) \\
&= &\omega(y)-\omega(x)-\langle x^*, y-x\rangle.
\end{eqnarray*}
It remains to show that $x^*\in \partial \omega(x)$, or equivalently, $x=\nabla \omega^*(x^*)$. In fact, noting that $x^*_{k_l}\in \partial \omega(x_{k_l})$, we derive that
\begin{eqnarray*}
\|\nabla\omega^*(x^*)-x \|&= & \lim_{l\rightarrow \infty}\|\nabla\omega^*(x^*)-x_{k_l}\| \\
  &= &\lim_{l\rightarrow \infty}\|\nabla\omega^*(x^*)-\nabla\omega^*(x_{k_l}^*)\| \\
&\leq &\lim_{l\rightarrow \infty}\frac{1}{\mu}\|x^*-x_{k_l}^*\|=0,
\end{eqnarray*}
where the inequality follows from the gradient-Lipschitz-continuity of $\omega^*$ stated in Lemma \ref{scLip}.
\end{proof}

\begin{theorem}[Convergence to an inner-level solution]\label{convinner}
Let $\{x_k\}$  be generalized by Algorithm \ref{ITS} with $t_k$ being an exact solution to the dual problem \eqref{D00} or a constant in $(0,\frac{2\mu}{L})$. Then, there exists a point $x$ in $\overline{X}$ such that $x_k\rightarrow x$ as $k\rightarrow +\infty$.
\end{theorem}
\begin{proof}
  By Lemma \ref{lembasic}, for any fixed $x\in \overline{X}$ we have that $\{D_\omega^{x^*_k}(x,x_k)\}$ monotonically decreases and hence must be a convergence sequence. Combining with the strong convexity of $\omega$ yields that
  $$\frac{\mu}{2}\|x-x_k\|^2\leq D_\omega^{x^*_k}(x,x_k) \leq D_\omega^{x^*_0}(x,x_0),$$
  which means that $\{x_k\}$ is bounded. It suffices to show that it has a unique cluster point. To this end, let $\{x_{k_i}\}$ and $\{x_{k_j}\}$ be two subsequences of $\{x_k\}$ such that $x_{k_i}\rightarrow \hat{x}$ and $x_{k_j}\rightarrow \tilde{x}$; it reduces to show that $\hat{x}=\tilde{x}$. In fact, invoking Lemma \ref{lemmon}, we have that
  $$\lim_{i\rightarrow \infty}D_\omega^{x_{k_i}^*}(x,x_{k_i})=D_\omega^{\hat{x}^*}(x,\hat{x})$$
  and 
    $$\lim_{j\rightarrow \infty}D_\omega^{x_{k_j}^*}(x,x_{k_j})=D_\omega^{\tilde{x}^*}(x,\tilde{x}).$$
    Due to the convergence of $\{D_\omega^{x^*_k}(x,x_k)\}$, we immediately get 
   
   \begin{equation}\label{eq001}
   D_\omega^{\hat{x}^*}(x,\hat{x})=D_\omega^{\tilde{x}^*}(x,\tilde{x}).
   \end{equation}
On the other hand, from Lemma \ref{lembasic} we can get 
$$\frac{\mu}{2L^2}\|\nabla f(x_k)\|^2 \leq  D_\omega^{x_{k}^*}(x,x_{k})- D_\omega^{x_{k+1}^*}(x,x_{k+1}).$$
Using the convergence of $\{D_\omega^{x^*_k}(x,x_k)\}$  and letting $k$ tend to $\infty$, we get $\lim_{k\rightarrow +\infty}\|\nabla f(x_k)\|=0$ and hence $\|\nabla f(\hat{x})\|=\lim_{i\rightarrow +\infty}\|\nabla f(x_{k_i})\|=0$ and $ \|\nabla f(\tilde{x})\|=\lim_{j\rightarrow +\infty}\|\nabla f(x_{k_j})\|=0$. Thus, both $\hat{x}$ and $\tilde{x}$ belong to $\overline{X}$. Now, we can take $x=\hat{x}$ in \eqref{eq001} to get $D_\omega^{\tilde{x}^*}(\hat{x},\tilde{x})=0$ which implies that $\hat{x}=\tilde{x}$ by the strong convexity of $\omega$. This completes the proof. 
\end{proof}

\subsection{Convergence to the optimal solution}

\begin{theorem}[Convergence to the optimal solution]\label{mainth1}
 Assume that $f$ is in the form of $f(x)=g(Ax)$, where $A\in\RR^{m\times n}$ is a nonzero matrix and $g:\RR^m\rightarrow \RR$ is a differentiable convex function satisfying one of the following conditions:
\begin{enumerate}
  \item[(i)] $A$ is surjective and $g$ has a unique minimizer;
  \item[(ii)] $g$ is a strictly convex function.
\end{enumerate}
Let $\{x_k\}$  be generalized by Algorithm \ref{ITS} with $t_k$ being an exact solution to the dual problem \eqref{D00} or a constant in $(0,\frac{2\mu}{L})$, and with initial points $x_0$ and $x_0^*$ satisfying $x_0^*\in \partial \omega(x_0)\bigcap \mathcal{R}(A^T).$
 Then, the sequence $\{x_k\}$  converges to the optimal solution $\bar{x}$ of the bilevel optimization problem \eqref{BilevelOpt}.
\end{theorem}
\begin{proof}
  By Theorem \ref{convinner}, we know that there exists $\tilde{x}\in\overline{X}$ such that $x_k\rightarrow \tilde{x}$ as $k\rightarrow 
  \infty$. Hence, $\{x_k\}$ must be a bounded sequence, and so is the subgradient sequence $\{x_k^*\}$ by \cite[Theorem 3.16]{beck2017first} Moreover, based on the iterative scheme of $x_k^*$, we deduce that 
  $$x_{k+1}^*=x_k^*-t_k\nabla f(x_k)=\cdots=x_0^*-\sum_{i=0}^{k}t_i\nabla f(x_i).$$
  Since $x_0^*\in \partial \omega(x_0)\bigcap \mathcal{R}(A^T)$, there exists a vector $y_0\in\RR^m$ such that 
  $x_0^*=A^Ty_0$. Due to the expression of $f(x)=g(Ax)$, we have $\nabla f(x)=A^T\nabla g(Ax)$. Therefore, 
  $$x_{k+1}^*=A^Ty_0-\sum_{i=0}^{k}t_iA^T\nabla g(Ax_i)=A^T(y_0-\sum_{i=0}^{k}t_i\nabla g(Ax_i)).$$
  Let $z_k=y_0-\sum_{i=0}^{k}t_i\nabla g(Ax_i)$ and decompose it into the direct sum of $z_k=z_k^1+z_k^2$ with $z_k^1\in \mathcal{K}(A^T)$ and $z_k^2\in \mathcal{R}(A)$. Then, we have 
  $$x_{k+1}^*=A^Tz_k=A^T(z_k^1+z_k^2)=A^Tz_k^2.$$
  Note that $\{x_k^*\}$ is a bounded sequence and $A^T$ is one-to-one from $\mathcal{A}$ to $\RR^n$. The sequence $\{z_k^2\}$ must be bounded, that is, there exists a constant $C>0$ such that 
  $$\|z_k^2\|\leq C, ~~\forall~k\geq 0.$$
  Now, using the subgradient inequality and Cauchy-Schwartz inequality, we derive that for $k\geq 1$,
  \begin{eqnarray} \label{lin1}
\omega (x_k)&\leq & \omega(\bar{x})-\langle A^Tz_{k-1}^2, \bar{x}-x_k\rangle \nonumber\\
&=&\omega(\bar{x})-\langle z_{k-1}^2, A\bar{x}-Ax_k\rangle \nonumber\\
&\leq&\omega(\bar{x})+\|z_{k-1}^2\|\cdot\|A\bar{x}-Ax_k\|\nonumber\\
&\leq&\omega(\bar{x})+C\cdot\|A\bar{x}-Ax_k\|.   
\end{eqnarray}
In what follows, under the two settings of $g$ and $A$, we will show that 
\begin{equation}\label{remain}
A\hat{x}=A\bar{x}, ~~\forall~ \hat{x}\in\overline{X},
\end{equation}
which immediately implies that 
$$\|A\bar{x}-Ax_k\|=\|A\tilde{x}-Ax_k\|\leq \|A\|\cdot\|\tilde{x}-x_k\|\rightarrow 0,$$
as $k\rightarrow \infty$ since  $x_k\rightarrow \tilde{x}$ as $k\rightarrow \infty$. Together with \eqref{lin1}, we get 
$\omega(\tilde{x})=\lim_{k\rightarrow \infty}\omega(x_k)\leq \omega(\bar{x}),$ which means that $\tilde{x}$ is also a minimizer of the bilevel optimization problem \eqref{BilevelOpt} and hence equals to $\bar{x}$ by the uniqueness of solution.  In other words, the sequence $\{x_k\}$  converges to the optimal solution $\bar{x}$ of the bilevel optimization problem \eqref{BilevelOpt}. Therefore, it remains to show \eqref{remain}. In the case (i), we denote the unique minimizer of $g$ by $\bar{y}$. For any $\hat{x}\in\overline{X}$, we have 
$$0=\nabla f(\hat{x})=A^T\nabla g(A\hat{x}),$$
which implies that $\nabla g(A\hat{x})=0$ due to the fact of $A$ being surjective. Thus, $A\hat{x}$ is a global minimizer of $g$ due to the convexity of $g$ and hence it must equal to the unique minimizer $\bar{y}$, that is, $A\hat{x}=\bar{y}$. Therefore, the result \eqref{remain} holds in the first case. In the case (ii), we will use the following property implied by the strict convexity of $g$:
\begin{equation}\label{stc}
\langle \nabla g(u)-\nabla g(v), u-v\rangle >0, \forall u\neq v.
\end{equation}
Suppose that there exists $\hat{x}\in\overline{X}$ such that $A\hat{x}\neq A\bar{x}$; then by \eqref{stc}, we have 
$$\langle \nabla g(A\hat{x})-\nabla g(A\bar{x}), A\hat{x}-A\bar{x}\rangle >0.$$
On the other hand, we have 
$$A^T\nabla g(A\hat{x})=A^T\nabla g(A\bar{x})=0,$$
and hence 
\begin{eqnarray}  
&&\langle \nabla g(A\hat{x})-\nabla g(A\bar{x}), A\hat{x}-A\bar{x}\rangle\nonumber\\
&=&\langle A^T\nabla g(A\hat{x})-A^T\nabla g(A\bar{x}), \hat{x}-\bar{x}\rangle=0 \nonumber.
\end{eqnarray}
It is a contradiction.  This completes the proof.  
\end{proof}
Condition (i) in Theorem \ref{mainth1} is weaker than the corresponding assumption on $A$ and $g$ in Theorem 4.1 in \cite{zhang2023revisiting} since the growth property there must imply the uniqueness of the minimizer of $g$. Regarding condition (ii), it is general enough to apply to the first example in Section 3 because the inner-level objective function $f(x)=\frac{1}{2}\|Ax-b\|^2$ can be written as $f(x)=g(Ax)$ with 
$$g(y)=\frac{1}{2}\|y-b\|^2,$$
which is a strictly convex function and satisfies the assumption on $f$ in Theorem \ref{mainth1}. Therefore, we have the following corollary, which is the main contribution made in a series of papers \cite{cai2009linearized,cai2009convergence,cai2010singular}.

\begin{corollary}\label{coro001}
  Consider the bilevel optimization problem \eqref{BilevelOpt} with $\omega(x)=\lambda\|x\|_1+\frac{1}{2}\|x\|_2$ and $f(x)=\frac{1}{2}\|Ax-b\|^2$, where $\lambda>0$ is some fixed parameter, $A\in\RR^{m\times n}$ is a given nonzero matrix, and $b\in\RR^m$ is a given vector. Let $\{x_k\}$  be generalized by the following iterative scheme:
    \begin{eqnarray}\label{Lbi}
  \left\{\begin{array}{rll}
x_{k+1} &= & \mathcal{S}_\lambda(x_k^*-t_k\cdot A^T(Ax_k-b)), \\
x_{k+1}^*&=& x_k^*-t_k\cdot A^T(Ax_k-b), 
\end{array} \right.
\end{eqnarray}
 initialized with $x_0=x_0^*=0$,  where $\mathcal{S}_\lambda(x)=\min(|x|-\lambda,0)\textrm{sign}(x)$ is the componentwise soft shrinkage, and $t_k$ is a constant in $\left(0, \frac{2}{\|AA^T\|}\right)$ or an exact solution to the dual problem 
  
  \begin{equation}\label{dual001}
    \Min_{t>0} g(t)=\frac{1}{2}\|\mathcal{S}_\lambda(x_k^*-t\cdot A^T(Ax_k-b))\|^2+t\cdot\beta_k,
  \end{equation}
   with $\beta_k=\langle A^T(Ax_k-b), x_k\rangle -\frac{1}{\|AA^T\|}\|A^T(Ax_k-b)\|^2$.
 Then, the sequence $\{x_k\}$  converges to the optimal solution $\bar{x}$ of the bilevel optimization problem \eqref{BilevelOpt}.
\end{corollary}

The proof follows directly from Theorem \ref{mainth1} and the fact that
$$\nabla \omega^*(x^*)=\mathcal{S}_\lambda(x^*), ~~\omega^*(x^*)=\frac{1}{2}\|\mathcal{S}_\lambda(x^*)\|^2.$$
The iterative scheme \eqref{Lbi} is precisely the linearized Bregman method up to the step-size $t_k$; please see (63d)-(63e) in \cite{lai2012augmented}. In original linearized Bregman methods, the step-size is usually set to be the constant $1/\|A\|^2$ to deduce convergence; see \cite{cai2009linearized}. Recently, it has been allowed to be the exact solution to the dual problem \eqref{dual001}, but in this setting the convergence to the optimal solution $\bar{x}$ is absent; see \cite{lorenz2014linearized}.  Our result in the corollary covers and supplements these existing studies. 

Let us turn to the problem \eqref{CFP}. For simplicity, we restrict our attention to the case of $N=1$, that is, 
$$f(x)=\frac{1}{2}\dist^2(Ax, Q),$$
where $Q\subset \RR^m$ is a given closed convex set. This form is general enough to cover the different noise models 
\begin{equation}\label{model}
 \Min_x \omega(x), ~\st ~~\|Ax-b\|_\diamond \leq \sigma,
\end{equation}
where the choice of the norm $\|\cdot\|_\diamond$ depends on the characteristics of noise involved in measurement vector $b$; e.g., we use $\ell_2$-norm for Gaussian noise, $\ell_1$-norm for impulsive noise and $\ell_\infty$-norm for uniformly distributed noise. Let $Q:=\{y\in\RR^m: \|y-b\|_\diamond \leq \sigma\}$; then the constraint $\|Ax-b\|_\diamond \leq \sigma$ can be written as $f(x)=\dist^2(Ax, Q)$ and hence the model \eqref{model} can be written as a bilevel convex optimization \eqref{BilevelOpt} so that our proposed  Algorithm \ref{ITS} can be applied. In particular, the model \eqref{model} with $\omega(x)=\lambda\|x\|_1+\frac{1}{2}\|x\|_2$ was discussed in \cite{lorenz2014linearized},  with the following iterative scheme 
    \begin{eqnarray}\label{Lbi01}
  \left\{\begin{array}{rll}
x_{k+1} &= & \mathcal{S}_\lambda(x_k^*-t_kA^T(Ax_k-\mathcal{P}_Q(Ax_k))), \\
x_{k+1}^*&=& x_k^*-t_kA^T(Ax_k-\mathcal{P}_Q(Ax_k)), 
\end{array} \right.
\end{eqnarray}
where $\mathcal{P}_Q$ is the projection operator onto $Q$. This scheme can be deduced from our proposed 
 Algorithm \ref{ITS} as well and hence its convergence to a feasible point and a sublinear rate of convergence for the inner-level function can be established by Theorem \ref{convinner} and Theorem \ref{thrate} respectively. However, the answer to the question of whether the iterates $\{x_k\}$ converges to the optimal solution $\bar{x}$ is negative even when the matrix $A$ is surjective. On one hand, the function $f(x)=g(Ax)$ with $g(y)=\frac{1}{2}\dist^2(y,Q)$ fails to satisfy the conditions (i) and (ii) in Theorem \ref{mainth1}; on the other hand, we give numerical tests to illustrate that the iterates $\{x_k\}$ may converge to a different point from the optimal solution $\bar{x}$. 

\subsection{Linear and sublinear convergence}
\begin{theorem}\label{thrate}
Let $\{x_k\}$  be generated by Algorithm \ref{ITS} with $t_k$ belonging to $(0,\frac{2\mu}{L}]$. Then the sequence of function values $\{f(x_k)\}$ is monotonically nonincreasing.  If the step-size $t_k$ is restricted to the interval $(0,\frac{\mu}{L}]$, then the sequence of function values $\{f(x_k)\}$ monotonically converges to the optimal function value $\bar{f}:=\min f(x)$ sublinearly in the sense that 
\begin{equation}\label{sublinear}
  f(x_{T+1})-\bar{f}\leq \frac{\inf_{x\in\overline{X}}D_\omega^{x_0^*}(x,x_0)}{\sum_{k=0}^{T}t_k}. 
\end{equation}
If the step-size $t_k$ is allowed to belong to a larger interval $(0,\frac{2\mu}{L})$ and the out-level objective $\omega$ is assumed to be gradient-Lipschtiz-continuous with constant $\nu>0$, then we have that 
\begin{equation}\label{sublinear0}
  f(x_{T+1})-\bar{f}\leq \frac{f(x_0)-\bar{f}}{1+ (f(x_0)-\bar{f})\cdot\sum_{k=0}^{T}h(t_k)},
\end{equation}
where $h(\tau)=\frac{\tau(2\mu-L\tau)\mu}{4\nu\cdot c(x_0)}$ with $c(x_0)=\inf_{x\in\overline{X}}D_\omega^{x_0^*}(x,x_0).$
\end{theorem}
\begin{proof}
By the three-point identity \eqref{Bregdis1} of Bregman distance in Lemma \ref{lemBreg}, we can derive that 
\begin{eqnarray*} 
&&D_\omega^{x_k^*}(x,x_k) - D_\omega^{x_{k+1}^*}(x,x_{k+1}) - D_\omega^{x_k^*}(x_{k+1},x_k)\\
&=&\langle x-x_{k+1}, x_{k+1}^*-x_k^*\rangle\\
&=& -t_k\langle x-x_{k+1}, \nabla f(x_k)\rangle\\
&=& -t_k(\langle x-x_{k}, \nabla f(x_k)\rangle+\langle x_k-x_{k+1}, \nabla f(x_k)\rangle).
\end{eqnarray*}
Using the subgradient inequality, we have 
$$\langle x-x_{k}, \nabla f(x_k)\rangle\leq f(x)-f(x_k).$$
Using the Lipschitz continuity of $\nabla f$, we have 
$$\langle x_k-x_{k+1}, \nabla f(x_k)\rangle\leq f(x_k)-f(x_{k+1})+\frac{L}{2}\|x_{k+1}-x_k\|^2.$$
Therefore, the term $D_\omega^{x_k^*}(x,x_k)- D_\omega^{x_{k+1}^*}(x,x_{k+1}) - D_\omega^{x_k^*}(x_{k+1},x_k)$ can be bounded as follows:
\begin{equation}\label{term1}
D_\omega^{x_k^*}(x,x_k)- D_\omega^{x_{k+1}^*}(x,x_{k+1}) - D_\omega^{x_k^*}(x_{k+1},x_k)\geq t_k(f(x_{k+1})-f(x))-\frac{Lt_k}{2}\|x_{k+1}-x_k\|^2.
\end{equation}
Letting $x=x_k$ in \eqref{term1} and using the relationship \eqref{Bregdis3} in Lemma \ref{lemBreg}, we can derive that 
\begin{eqnarray} \label{more7}
t_k(f(x_{k+1})-f(x_k))&\leq &-D_\omega^{x_k^*}(x_{k+1},x_k) - D_\omega^{x_{k+1}^*}(x_k,x_{k+1}) +\frac{Lt_k}{2}\|x_{k+1}-x_k\|^2\nonumber\\
&\leq&-\mu\|x_{k+1}-x_k\|^2+\frac{Lt_k}{2}\|x_{k+1}-x_k\|^2\nonumber\\
&=& \left(\frac{Lt_k}{2}-\mu\right)\|x_{k+1}-x_k\|^2.
\end{eqnarray}
If the step-size $t_k\in (0,\frac{2\mu}{L}]$, then $\left(\frac{Lt_k}{2}-\mu\right)\leq 0$ and hence $f(x_{k+1})\leq f(x_k)$ which means that the sequence of function values $\{f(x_k)\}$ is monotonically nonincreasing.

Now, we return to \eqref{term1} and derive that for $t_k\in (0,\frac{\mu}{L}]$,
\begin{eqnarray*} 
t_k(f(x_{k+1})-f(x))&\leq &D_\omega^{x_k^*}(x,x_k) - D_\omega^{x_{k+1}^*}(x,x_{k+1}) - D_\omega^{x_k^*}(x_{k+1},x_k) +\frac{Lt_k}{2}\|x_{k+1}-x_k\|^2\\
&\leq&D_\omega^{x_k^*}(x,x_k) - D_\omega^{x_{k+1}^*}(x,x_{k+1})+\left(\frac{Lt_k}{2}-\frac{\mu}{2}\right)\|x_{k+1}-x_k\|^2\\
&\leq&D_\omega^{x_k^*}(x,x_k) - D_\omega^{x_{k+1}^*}(x,x_{k+1}).
\end{eqnarray*}
Summing the inequality above from $k=0$ to $k=T$, we obtain
\begin{equation}\label{term2}
\sum_{k=0}^{T}t_k(f(x_{k+1})-f(x))\leq D_\omega^{x_0^*}(x,x_0) - D_\omega^{x_{T+1}^*}(x,x_{T+1})\leq D_\omega^{x_0^*}(x,x_0).
\end{equation}
By the monotonicity of $\{f(x_k)\}$, we can get 
\begin{equation}\label{term3}
(f(x_{T+1})-f(x))\sum_{k=0}^{T}t_k \leq D_\omega^{x_0^*}(x,x_0),
\end{equation}
which implies the desired result \eqref{sublinear} by noting that $\inf_{x\in\overline{X}}(f(x_{T+1})-f(x))=f(x_{T+1})-\bar{f}.$

Below, let us show \eqref{sublinear0}. 
Since the out-level objective $\omega$ is additionally assumed to be gradient-Lipschtiz-continuous with constant $\nu>0$, we can derive that 
\begin{eqnarray*} 
\|x_{k+1}-x_k\|&\geq &\nu^{-1}\|\nabla \omega(x_{k+1})-\nabla\omega x_k\|\\
&=&\nu^{-1}\|x_{k+1}^*-x_k^*\|\\
&=&\nu^{-1}t_k\|\nabla f(x_k)\|.
\end{eqnarray*}
Combining this with the inequality \eqref{more7} and noting the step-size condition $t_k\in (0,\frac{2\mu}{L})$, we can get 
\begin{equation}\label{more8}
f(x_{k+1})-f(x_k)\leq t_k\nu^{-2}\left(\frac{Lt_k}{2}-\mu\right)\|\nabla f(x_k)\|^2.
\end{equation}
On the other hand, using the subgradient inequality, the Cauchy-Schwartz inequality, and the strong convexity of $\omega$, we derive that for any $x\in\overline{X}$,
\begin{eqnarray} \label{more9}
f(x_{k})-f(x)&\leq &\langle \nabla f(x_k), x_k-x\rangle \nonumber\\
&\leq&\|\nabla f(x_k)\| \cdot \| x_k-x\|\nonumber\\
&\leq& \|\nabla f(x_k)\|\cdot \sqrt{\frac{2}{\mu}D_\omega^{x_k}(x,x_k)}\nonumber\\
&\leq& \|\nabla f(x_k)\| \cdot \sqrt{\frac{2}{\mu}D_\omega^{x_0}(x,x_0)},
\end{eqnarray}
where the last inequality follows from the descent property \eqref{desc2} in Lemma \ref{lembasic}. Thus, we can get 
\begin{equation}\label{more0}
f(x_{k})-\bar{f}\leq \|\nabla f(x_k)\| \cdot \sqrt{\frac{2}{\mu}\inf_{x\in\overline{X}}D_\omega^{x_0}(x,x_0)}=\sqrt{\frac{2c(x_0)}{\mu}}\cdot\|\nabla f(x_k)\|.
\end{equation}
For simplicity of deduction, we let $f_k:=f(x_k)-\bar{f}$. Then, combining \eqref{more8} and \eqref{more0} and using the definition of $h(\cdot)$, we have 
\begin{equation}\label{more1}
f_{k+1}-f_k\leq t_k\nu^{-2}\left(\frac{Lt_k}{2}-\mu\right)\frac{\mu}{2c(x_0)}\cdot f_k^2=-h(t_k)\cdot f_k^2.
\end{equation}
Thus, it holds that
\begin{equation}\label{more2}
\frac{1}{f_{k+1}}\geq \frac{1}{f_k}+h(t_k)\cdot\frac{f_k}{f_{k+1}}\geq \frac{1}{f_k}+h(t_k),
\end{equation}
where the last inequality follows from the monotonicity of $\{f(x_k)\}$. Summing \eqref{more2} from $k=0$ to $T$, we get 
\begin{equation}\label{more3}
\frac{1}{f_{T+1}}\geq \frac{1}{f_0}+\sum_{k=0}^{T} h(t_k).
\end{equation}
Therefore, we have 
\begin{equation}\label{more4}
f_{T+1}\leq \frac{f_0}{1+ f_0\cdot \sum_{k=0}^{T} h(t_k)}.
\end{equation}
This completes the proof. 
\end{proof}

Note that when $\omega(x) = \frac{1}{2}\|x\|^2$, $\nabla \omega^*$ is a identity mapping and Algorithm \ref{ITS} reduce to the gradient descent method. In this case, the allowed step-size interval and the sublinear convergence result are consistent with the classical result \cite{bubeck2015convex} for the gradient descent method. This implies that  \eqref{sublinear0} in Theorem \ref{thrate} provides a tight bound. 

In order to derive rates of linear convergence, we introduce the following Bregman distance growth condition, which was independently proposed in \cite{zhang2021proximal} and \cite{lu2018relatively} recently. 

\begin{definition}
Let $\overline{X}$ be the minimizer set of $f$ and $\bar{f}$ be its optimal function value. We say that $f$ satisfies the Bregman distance growth condition if there exists a constant $\gamma>0$ such that 
\begin{equation}\label{Bregcond}
  f(x)-\bar{f}\geq \gamma \cdot \inf_{z\in\overline{X}}D_\omega^{x^*}(z,x). 
\end{equation}
\end{definition}

\begin{definition}
Let $\overline{X}$ be the minimizer set of $f$ and $\bar{f}$ be its optimal function value. We say that $f$ satisfies the quadratic growth condition if there exists $\eta>0$ such that for all $x\in \RR^n$ we have  
\begin{equation}\label{QGcond}
  f(x)-\bar{f}\geq \frac{\eta}{2}\dist^2(x,\overline{X}). 
\end{equation}
\end{definition}

\begin{theorem}[linear convergence]\label{mainth01}
Let $\{x_k\}$  be generated by Algorithm \ref{ITS} with $t_k$ being an exact solution to the dual problem \eqref{D00} or being the constant $\frac{\mu}{L}$. Assume that $f$ satisfies the Bregman distance growth condition \eqref{Bregcond}. Then, the sequence $\{x_k\}$  converges to the feasible set $\overline{X}$ linearly in the following sense 
\begin{equation}\label{linconv1}
  \inf_{x\in\overline{X}} D_\omega^{x_k^*}(x,x_k)\leq \left(1-\left(\frac{\mu\cdot \gamma}{2L}\right)^2\right)^k\inf_{x\in\overline{X}} D_\omega^{x_0^*}(x,x_0). 
\end{equation}
\end{theorem}

\begin{proof}
  Starting with the following basic descent property stated in Lemma \ref{lembasic}, 
  \begin{equation*}
 D_\omega^{x_{k+1}^*}(x,x_{k+1})\leq  D_\omega^{x_{k}^*}(x,x_{k})-\frac{\mu}{2L^2}\|\nabla f(x_k)\|^2, ~\forall x\in H_k,
\end{equation*}
we can get  
  \begin{equation}\label{desc01}
 \inf_{x\in\overline{X}} D_\omega^{x_{k+1}^*}(x,x_{k+1})\leq  \inf_{x\in\overline{X}} D_\omega^{x_{k}^*}(x,x_{k})-\frac{\mu}{2L^2}\|\nabla f(x_k)\|^2, ~\forall x\in H_k,
\end{equation}
where we have used the fact $\overline{X}\subset H_k$. Using the Bregman distance growth condition \eqref{Bregcond} and the strong convexity of $\omega$, we can derive that 
  \begin{eqnarray} \label{desc001}
 f(x)-\bar{f}& \geq & \gamma \cdot \inf_{z\in\overline{X}}D_\omega^{x^*}(z,x) \nonumber \\
&\geq& \gamma \cdot \inf_{z\in\overline{X}} \frac{\mu}{2}\|z-x\|^2  \nonumber\\
&=& \frac{\gamma\cdot \mu}{2}\dist^2(x,\overline{X}).  
\end{eqnarray} 
 In other words, the inner objective function $f$ satisfies the quadratic growth condition and hence it also satisfies the PL inequality  
   \begin{equation}\label{PL0}
 \|\nabla f(x)\|\geq \tau\sqrt{f(x)-\bar{f}},
\end{equation}
where $\tau=\sqrt{\frac{\gamma\cdot\mu}{2}}$; see the interplay theorem in \cite{zhang2020new}.  
Now, combining \eqref{PL0}, \eqref{desc01}, and \eqref{Bregcond}, we derive that
  \begin{eqnarray} 
\inf_{x\in\overline{X}} D_\omega^{x_{k+1}^*}(x,x_{k+1})& \leq &  \inf_{x\in\overline{X}}D_\omega^{x_{k}^*}(x,x_{k}) -\frac{\mu}{2L^2}\tau^2(f(x_k)-\bar{f}) \nonumber \\
& \leq &  \inf_{x\in\overline{X}}D_\omega^{x_{k}^*}(x,x_{k}) -\frac{\mu}{2L^2}\tau^2\cdot\gamma \inf_{x\in\overline{X}}D_\omega^{x_{k}^*}(x,x_{k}) \nonumber\\
&=& \left(1-\left(\frac{\mu\cdot \gamma}{2L}\right)^2\right)\inf_{x\in\overline{X}} D_\omega^{x_{k}^*}(x,x_{k}).
\end{eqnarray} 
This completes the proof. 
\end{proof}

The key to obtaining linear convergence is the Bregman distance growth condition. Next, we will give the validation of the condition on problem \eqref{gBP}. We first give the following result from \cite{schopfer2019linear}. 
\begin{lemma}[\cite{schopfer2019linear}, Lemma 3.1]\label{lemmore1}
  Suppose that the linear system $Ax = b$ is consistent.  Let $\omega(x)=\lambda\|x\|_1+\frac{1}{2}\|x\|^2$, where $\lambda>0$ is fixed parameter. Denote
  $$\bar{x}=\arg\min_{x} \omega(x),~~\st~Ax=b.$$
  Then, there exists a constant $\tau>0$ such that for all $x\in\RR^n$ with $\partial \omega(x)\bigcap \mathcal{R}(A^T)\neq \emptyset$ and $x^*\in \partial \omega(x)\bigcap \mathcal{R}(A^T)\neq \emptyset$ we have 
\begin{equation}\label{EB}
D_\omega^{x^*}(\bar{x},x)\leq \tau\cdot\|Ax-b\|^2.
\end{equation}
\end{lemma}

\begin{corollary}\label{coro002}
  Under the same setting of Corollary \ref{coro001} and assuming that $A^Tb\neq 0$, then $f$ satisfies the Bregman distance growth condition with a constant $\frac{\eta}{2\tau L^2}$. Moreover, the sequence $\{x_k\}$  converges to the feasible set $\overline{X}$ linearly in the following sense that
\begin{equation}\label{linconv}
  \inf_{x\in\overline{X}} D_\omega^{x_k^*}(x,x_k)\leq \left(1-\left(\frac{\mu \eta }{4\tau L^3}\right)^2\right)^k \inf_{x\in\overline{X}} D_\omega^{x_0^*}(x,x_0). 
\end{equation}
\end{corollary}
\begin{proof}
  By the first-order optimality condition, the feasible set $\overline{X}$ can be written as 
  $$\overline{X}=\{x\in\RR^n: A^T(Ax-b)=0\}.$$
  Thus, the corresponding bilevel convex optimization problem \eqref{BilevelOpt} can be written as 
   $$\bar{x}=\arg\min_{x} \omega(x),~~\st~A^TAx=A^Tb.$$
  Using the basic fact that $\mathcal{R}(A^TA)=\mathcal{R}(A^T)$, we have that $A^Tb\neq 0$ and $A^Tb$ lies in $\mathcal{R}(A^TA)$. Therefore, invoking Lemma \ref{lemmore1} and using the fact $\mathcal{R}(A^TA)=\mathcal{R}(A^T)$ again, we conclude that there exists a constant $\tau>0$ such that for all $x\in\RR^n$ with $\partial \omega(x)\bigcap \mathcal{R}(A^T)\neq \emptyset$ and $x^*\in \partial \omega(x)\bigcap \mathcal{R}(A^T)\neq \emptyset$ we have 
\begin{equation}\label{EB01}
D_\omega^{x^*}(\bar{x},x)\leq \tau\cdot\|A^TAx-A^Tb\|^2.
\end{equation}
With this inequality, we further derive that 
\begin{equation}\label{EB02}
\inf_{z\in\overline{X}}D_\omega^{x^*}(z,x) \leq D_\omega^{x^*}(\bar{x},x)\leq \tau\cdot\|A^TAx-A^Tb\|^2= \tau\cdot\|\nabla f(x)\|^2. 
\end{equation}
Using the gradient Lipschitz continuity, we have that 
$$\|\nabla f(x)\|=\|\nabla f(x)-\nabla f(\hat{x})\|\leq L\cdot \|x-\hat{x}\|,~~\forall~\hat{x}\in\overline{X}.$$
Since that $f$ has the form of $g(Ax)$, it follows from \cite{necoara2019linear} that $f$ satisfies the quadratic growth condition, i.e., there exists a constant $\eta$ such that \eqref{QGcond} holds. This implies that 
$$\|\nabla f(x)\|\leq \inf_{\hat{x}\in\overline{X}} L\cdot \|x-\hat{x}\|=L\cdot \dist(x,\overline{X})\leq L\cdot\sqrt{\frac{2}{\eta}(f(x)-\bar{f})}.$$
 Therefore, combining with \eqref{EB02}, we obtain that 
$$\inf_{z\in\overline{X}}D_\omega^{x^*}(z,x) \leq \frac{2\tau L^2}{\eta}(f(x)-\bar{f}).$$
In other words, $f$ satisfies the Bregman distance growth condition and hence the conclusion follows from Theorem \ref{mainth01}. This completes the proof. 
\end{proof}

 In \cite{schopfer2019linear}, the authors provide a similar linear convergence result for the sparse Kaczmarz method, which is equivalent to the iterative scheme \eqref{Lbi}. It should be emphasized that they require that the linear system $Ax = b$ is consistent, as shown in Lemma \ref{lemmore1}. In our case, we drop this condition since we apply Lemma \ref{lemmore1} with the linear system $A^\top Ax = A^\top b$, which is always consistent. In contrast, we require that $A^\top b \neq 0$, this condition is more checkable compared with the condition in \cite{schopfer2019linear}.  

\section{Experiment results}\label{sec6}
In this section, we conduct four numerical experiments to illustrate different aspects of Algorithm \ref{ITS}. In Section \ref{sec:rbp} we investigate the use of the different step-sizes with a compressed sensing example.  In Section \ref{sec:fea} we show that it is simple to adapt Algorithm \ref{ITS} to feasibility problems by defining the distance function. In particular, we consider several sparse recovery problems under different types of noise, including Gaussian noise and uniform noise. Moreover, we will also answer the question of whether Algorithm \ref{ITS} computes an optimal solution when the constraint does not satisfy the condition in Theorem \ref{mainth1}.
\subsection{Linear inverse problem}\label{sec:rbp}
We conduct an experiment on the comparison of step-size rules for solving problem \eqref{gBP}. 
The gradient Lipschitz constant is $L: = \|A^\top A\|_2$. Our algorithm has the following iterative scheme: 
\begin{equation}\label{iter}
  \left\{
  \begin{aligned}
  x_{k+1}  & = \mathcal{S}_{\lambda}(x^*_k - t_k A^\top (Ax_k -b)),\\
  x_{k+1}^* & = x^*_k - t_k A^\top (Ax_k -b),
  \end{aligned}
  \right.
\end{equation}
where $t_k$ is the step-size, which can be chosen by the following strategy:
\begin{itemize}
  \item The exact step-size by the exact solution of the following dual problem
  \begin{equation}\label{pro:one}
    \Min_{t>0} \frac{1}{2}\|\mathcal{S}_{\lambda}(x^*_k - t A^\top (Ax_k -b))\|^2 + t\cdot \beta_k,
  \end{equation}
  where $\beta_k = \langle A^\top (Ax_k - b), x_k \rangle - \frac{\|A^\top (Ax_k - b)\|^2}{L} $. 
  \item The constant step-size $t_k = \frac{\mu}{L} = \frac{1}{L}$.
  \item The dynamic step-size $t_k = \frac{\|Ax_k - b\|_2^2}{\|A^\top (Ax_k - b)\|_2^2}$. 
\end{itemize}

In order to obtain an optimal solution we use the primal-dual method \cite{beck2017first}.  The method considers the following general optimization problem:
\begin{equation}
  \Min_{x\in \mathbb{R}^n} g(x) + h(Ax),
\end{equation}
where $g$ is strongly convex, closed, and proper. $h$ is properly closed and proximable. Employs the fast dual proximal gradient method whose general update step is
\begin{equation}\label{primal-dual}
\left\{
\begin{aligned}
u_k & =\operatorname{argmax}_u\left\{\left\langle u, A w_k\right\rangle-g(u)\right\}, \\
y_{k+1} & =w_k-\frac{1}{L_k} A u_k+\frac{1}{L_k} \operatorname{prox}_{\lambda L_k h}\left(A u_k-L_k w_k\right), \\
t_{k+1} & =\frac{1+\sqrt{1+4 t_k^2}}{2}, \\
w_{k+1} & =y_{k+1}+\left(\frac{t_k-1}{t_{k+1}}\right)\left(y_{k+1}-y_k\right).
\end{aligned}
\right.
\end{equation}
where $L_k$ is chosen either as a constant or by a backtracking procedure (default). The last computed vector $u_k$ is considered as the obtained solution and the last computed $y_k$ is the obtained dual vector. We apply the algorithm to solve \eqref{gBP} by letting $g(x) = \lambda \|x\|_1 + \frac{1}{2}\|x\|_2^2$ and $h(x) = \delta_{Q}(x)$, where $Q = \{x~:x=b\}$. 

For a given matrix $A$ we construct a recoverable sparse vector $x^{\dagger}$ using L1TestPack \cite{lorenz2012constructing}, compute the corresponding right-hand side $b = Ax^{\dagger}$, and choose the regularization parameter $\lambda = \|x^{\dagger}\|_1$. The results are summarized in Figure \ref{least}. We denote Algorithm \ref{ITS} with dynamic, exact, and constant step-sizes as ``Breg-dy'', ``Breg-opt'' and ``Breg-con'', respectively. The primal-dual method \eqref{primal-dual} is referred to as ``primal-dual''.  The reconstruction error is defined as $\|x_k - x^{\dagger}\|$, where $x_k$ is the iterate in Algorithm. The objective value denotes $\lambda\|x\|_1 + \frac{1}{2}\|x\|_2^2$. The feasibility is denoted by $\|Ax - b\|_2$. Figure \ref{least} (c) shows that All step-sizes have successfully attained feasible solutions, with the exact step size demonstrating linear convergence from the iterative point to the feasible solution, thereby validating Theorem \ref{mainth01}. 
Furthermore, all step-sizes have identified the optimal solution for problem \eqref{gBP}, as illustrated in Figure \ref{least} (b). This is attributed to the fulfillment of the conditions shown in Theorem \ref{mainth1} for problem \eqref{gBP}. 
Regarding the reconstruction error, it is observed that the exact step size exhibits a smaller reconstruction error compared to other step-sizes.


\begin{figure}[htbp]
\centering  
\subfigure[Reconstruction error ($\log$)]{
\includegraphics[width=5cm]{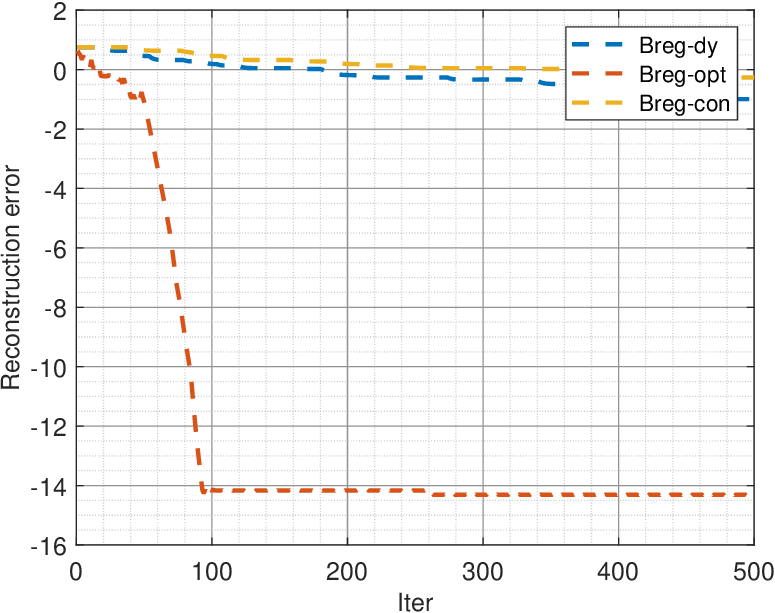}}
\subfigure[Objective value]{
\includegraphics[width=5cm]{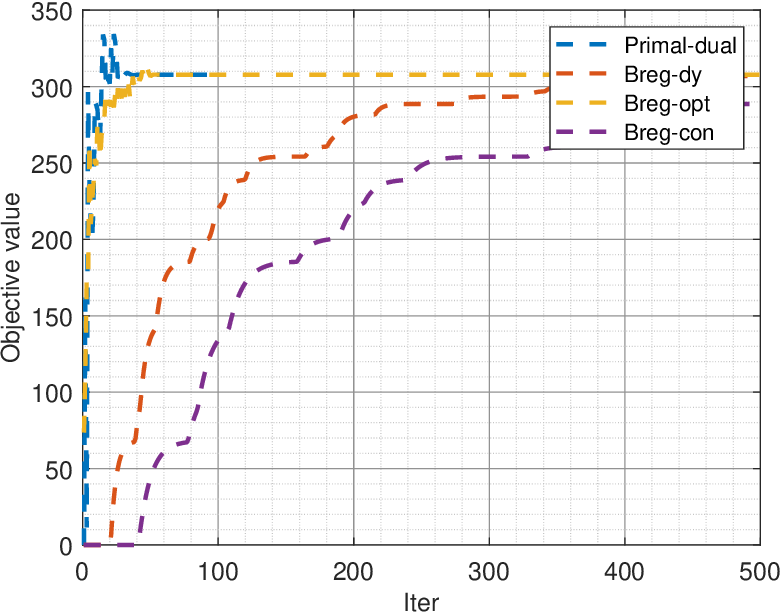}}
\subfigure[Feasibility ($\log$)]{
\includegraphics[width=5cm]{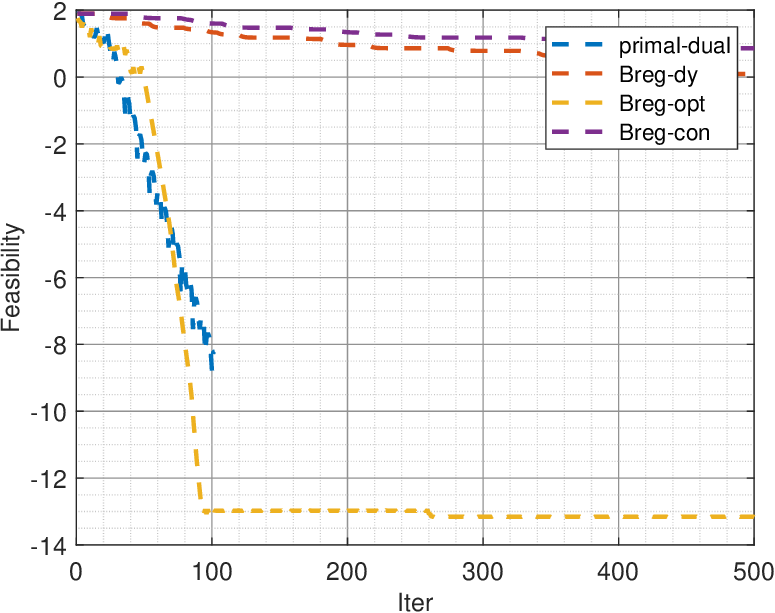}}
\caption{Comparison with different step-size for least square problem}
\label{least}
\end{figure}

\begin{figure}[htbp]
\centering  
\subfigure[True]{
\includegraphics[width=6cm]{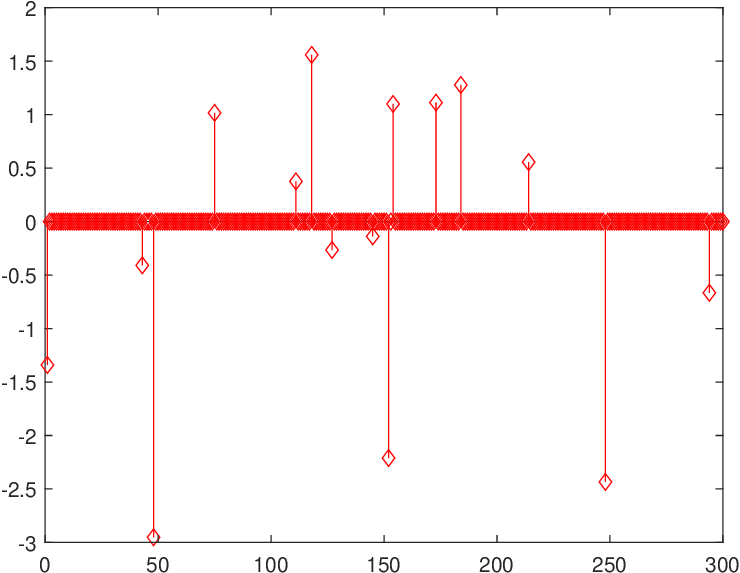}}
\subfigure[Constant]{
\includegraphics[width=6cm]{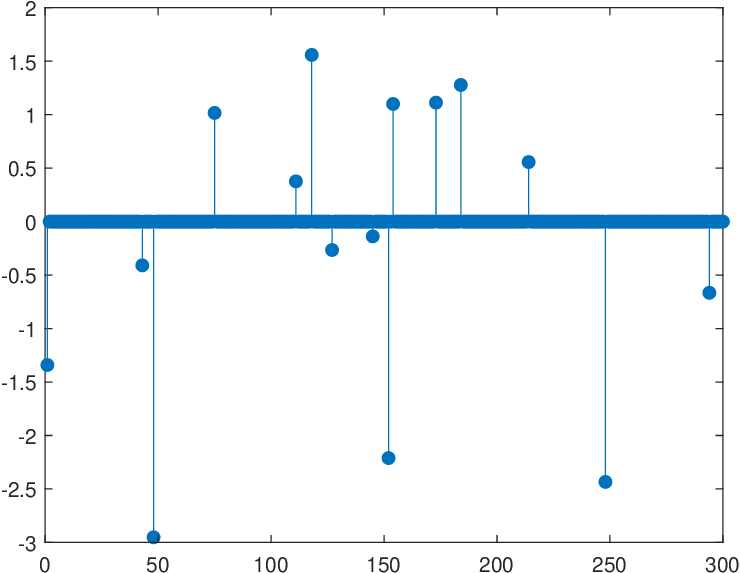}}
\subfigure[Optimal]{
\includegraphics[width=6cm]{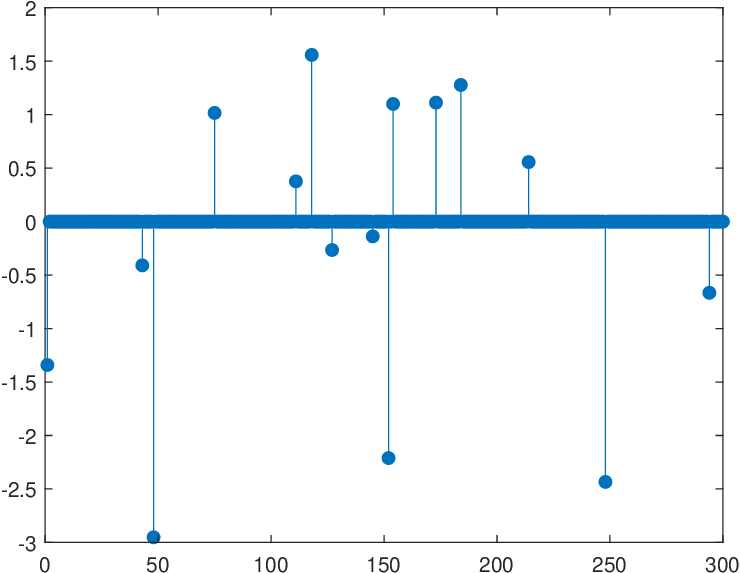}}
\subfigure[Primal-dual]{
\includegraphics[width=6cm]{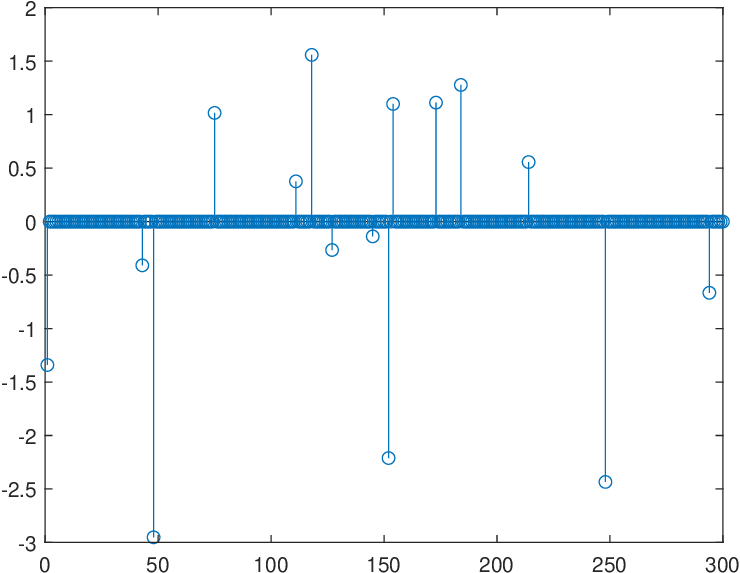}}
\caption{Comparison for least square problem}
\label{least1}
\end{figure}

\begin{figure}[htbp]
\centering  
\subfigure[Reconstruction error ($\log$)]{
\includegraphics[width=5cm]{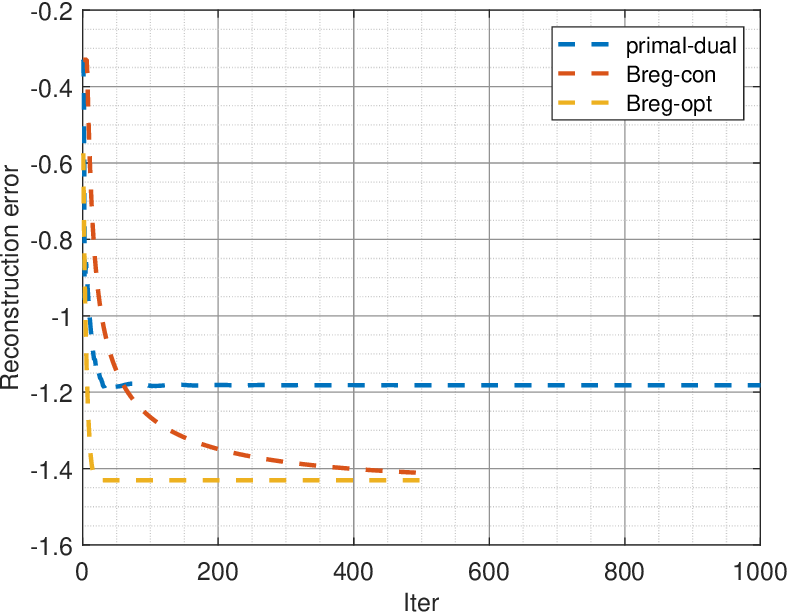}}
\subfigure[Objective value]{
\includegraphics[width=5cm]{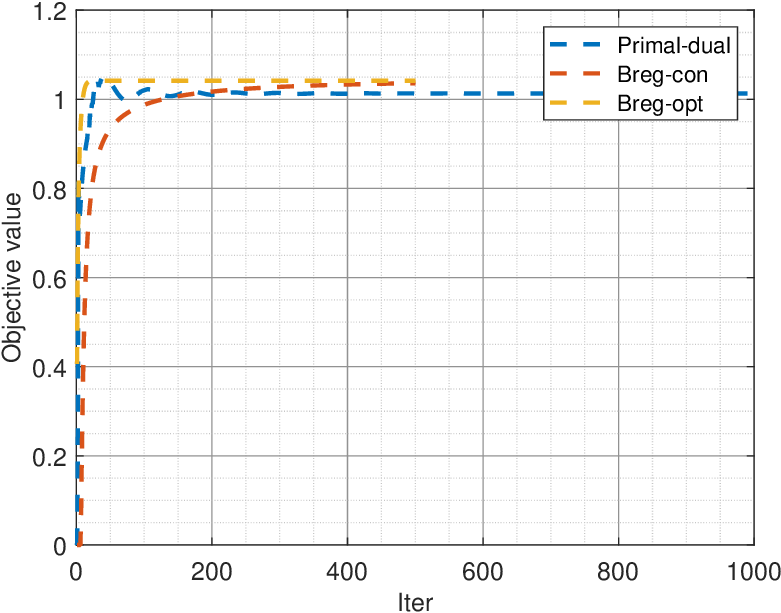}}
\subfigure[Feasibility ($\log$)]{
\includegraphics[width=5cm]{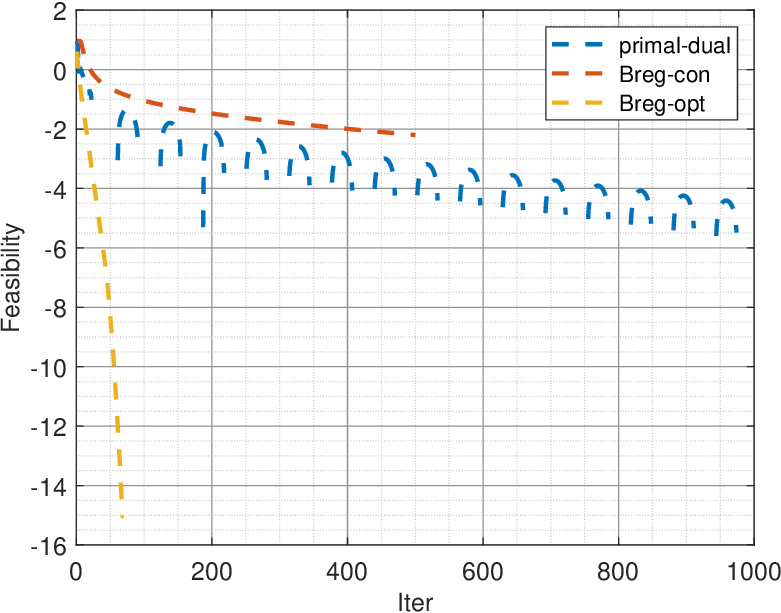}}
\caption{Comparison for the $\ell_2$ regularization problem}
\label{least-l2-1}
\end{figure}
\subsection{Feasibility problems}\label{sec:fea}
This subsection considers problem \eqref{CFP}. In particular, 
consider the following feasibility problem:
\begin{equation}\label{pro:fea}
  \Min_{x\in \mathbb{R}^n} \omega(x), \; \text{s.t.}\; Ax\in Q,
\end{equation}
where $\omega$ is strongly convex and $Q$ is an convex set. Denote $f(x): = \frac{1}{2}\mathbf{dist}^2(Ax,Q)$. Then one can construct an equivalent bilevel optimization problem as follows:
\begin{equation}\label{pro:bil:fea}
   \Min_{x\in \mathbb{R}^n} \omega(x), \; \text{s.t.}\; x\in \arg\min_x f(x) = \frac{1}{2}\mathbf{dist}^2(Ax,Q). 
\end{equation}
The gradient Lipschitz constant of $f$ is $L: = \|A^\top A\|_2$. Our algorithm has the following iterative scheme: 
\begin{equation}\label{iter11}
  \left\{
  \begin{aligned}
  x_{k+1}  & = \mathcal{S}_{\lambda}(x^*_k - t_k A^\top (Ax_k -\mathcal{P}_{Q}(Ax_k))),\\
  x_{k+1}^* & = x^*_k - t_k A^\top (Ax_k -\mathcal{P}_{Q}(Ax_k)).
  \end{aligned}
  \right.
\end{equation}
In order to illustrate the performance of Algorithm \ref{ITS} to solve problem \eqref{pro:bil:fea}, we consider the problem of recovering sparse solutions
of linear equations $Ax = b$, where only noisy data $b^{\sigma}$ is available for different noise models. In this case, the convex set $Q: = \{x~:\|x - b^{\sigma}\| \leq \sigma\}$, where $\sigma$ denote the noise. The choice of the norm $\|\cdot\|$ is dictated by the noise characteristics. In the following, we will consider three types of norms, i.e.,  the $\ell_2$-norm for Gaussian noise
and the $\ell_{\infty}$-norm for uniformly distributed noise.  In order to obtain the optimal solution, we also run the primal-dual algorithm to solve \eqref{pro:bil:fea} by letting $g(x) = \lambda \|x\|_1 + \frac{1}{2}\|x\|_2^2$ and $g(x) = \delta_{Q}(x)$. 
\subsubsection{Sparse recovery with Gaussian noise}
This subsection focus on an special case of problem \eqref{pro:bil:fea}, where $Q = \{y\in\mathbb{R}^m~:~\|y - b^{\sigma}\|_2 \leq \sigma\}$, $\sigma$ is the noise.
In this case
\begin{equation}\label{11}
  A x-\mathcal{P}_Q(A x)=\max \left\{0,1-\frac{\sigma}{\left\|A x-b^\sigma\right\|_2}\right\} \cdot\left(A x-b^\sigma\right).
\end{equation}
For some matrix $A\in \mathbb{R}^{m\times n}$ we produce a sparse vector $x^{\dagger}$ with only 30 nonzero entries and calculate $b = Ax^{\dagger}$. We choose $b^{\sigma}$ by adding the Gaussian noise to $b$, and $\sigma = \|b - b^{\sigma}\|_2$.

Figure \ref{least-l2-1} shows the result. Despite the successful identification of feasible solutions for all step-sizes, Figure \ref{least-l2-1} (b) illustrates that they yield larger objective function values compared to the primal-dual algorithm. This indicates that our algorithm may not necessarily find the optimal solution. The reason lies in the failure of problem \eqref{pro:bil:fea} to satisfy the conditions in Theorem \ref{mainth1}, thereby indirectly validating the reasonableness of these conditions.
Despite the absence of an optimal solution, Figure \ref{least-l2-1} (a) demonstrates that, under the exact step size, our algorithm accurately recovers the original solution. This further underscores the practical significance of our algorithm. In order to show the effect of the level of noise, we let the noise constant $\sigma = c\|b - b^{\sigma}\|_2$, where $c = 0.1, 0.5, 1$. Figure \ref{least-l2-2} shows that a low level of noise achieves a lower reconstruction error.

\begin{figure}[htbp]
\centering  
\subfigure[Reconstruction error ($\log$)]{
\includegraphics[width=5cm]{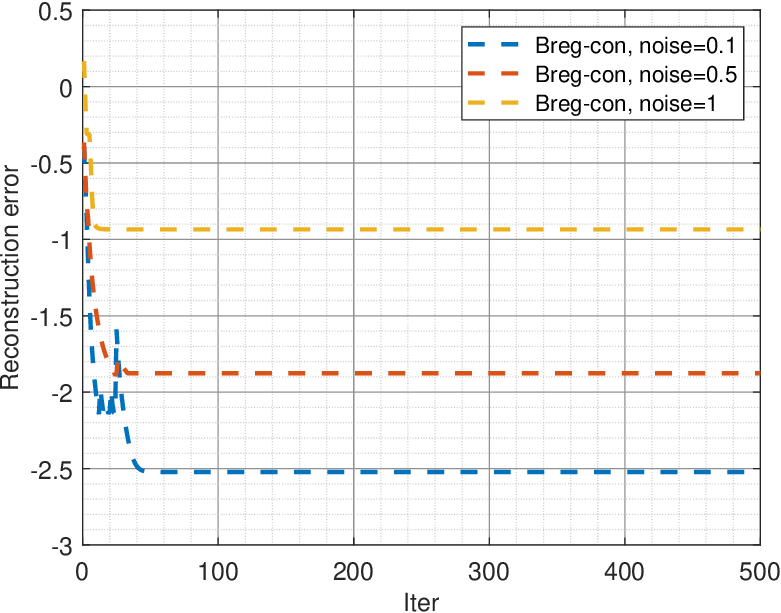}}
\subfigure[Objective value]{
\includegraphics[width=5cm]{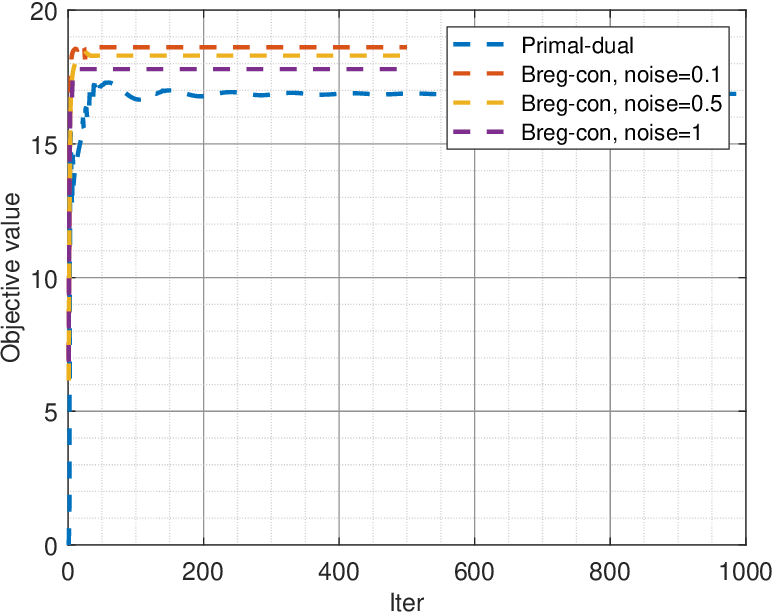}}
\subfigure[Feasibility ($\log$)]{
\includegraphics[width=5cm]{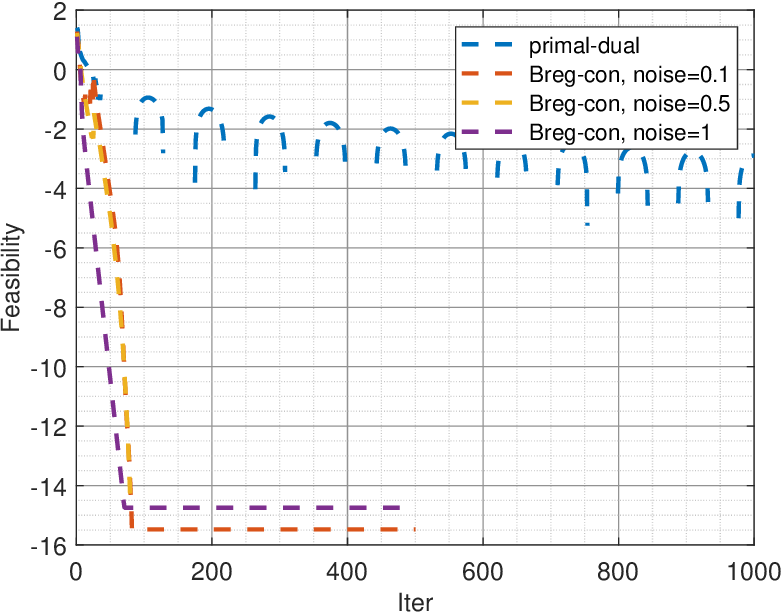}}
\caption{Comparison with different noise level for the $\ell_2$ regularization problem}
\label{least-l2-2}
\end{figure}

\subsubsection{Sparse recovery with uniform noise}
This subsection focus on an special case of problem \eqref{pro:bil:fea}, where $Q = \{y\in\mathbb{R}^m~:~\|y - b\|_{\infty} \leq \sigma\}$, $\sigma$ is the noise. In this case
\begin{equation}\label{12}
  A x-\mathcal{P}_Q(A x)=\mathcal{S}_{\sigma}(Ax - b^{\sigma}).
\end{equation}
For some matrix $A\in \mathbb{R}^{m\times n}$ we produce a sparse vector $x^{\dagger}$ with only 30 nonzero entries and calculate $b = Ax^{\dagger}$. We choose $b^{\sigma}$ by adding uniformly distributed noise with range $[-1,1]$, and $\sigma = \|b - b^{\sigma}\|_{\infty}$. Figure \ref{least-linf} shows a similar result as the previous subsection. Our algorithm with the exact step-size achieves better performance in terms of the reconstruction error and feasibility. 
\begin{figure}[htbp]
\centering  
\subfigure[Reconstruction error ($\log$)]{
\includegraphics[width=5cm]{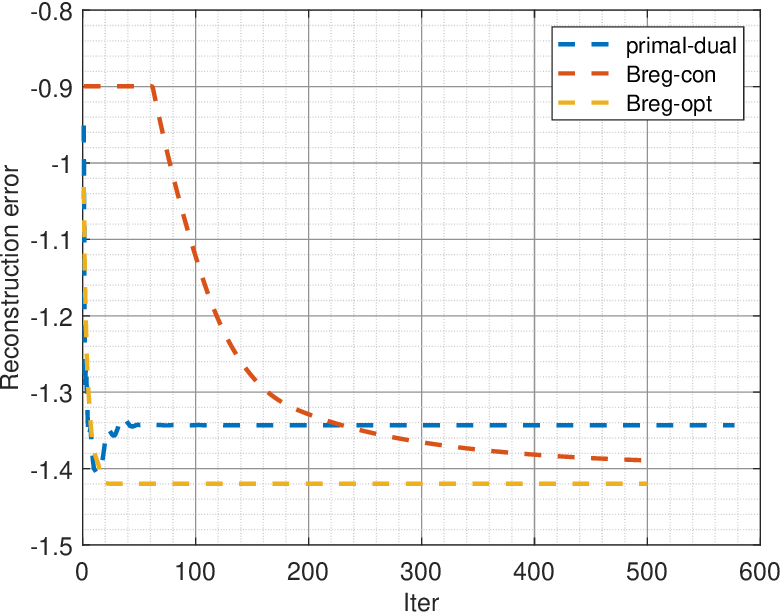}}
\subfigure[Objective value]{
\includegraphics[width=5cm]{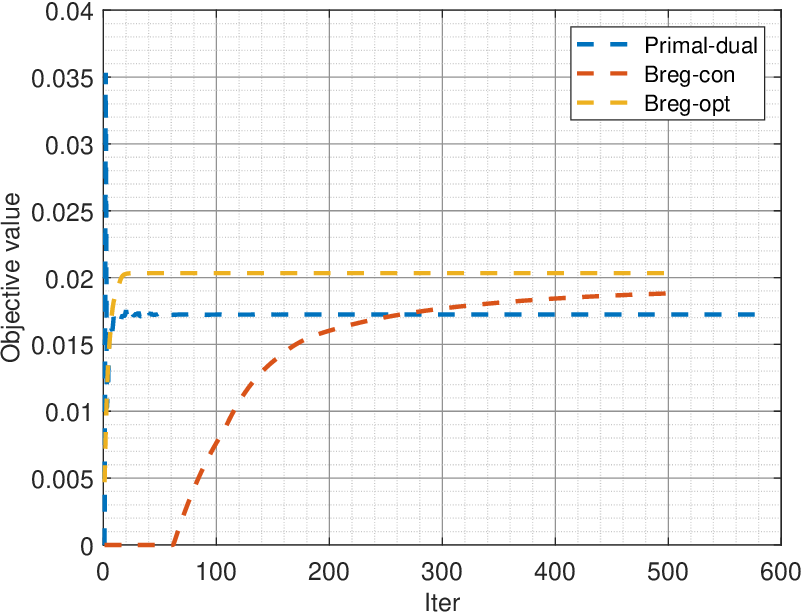}}
\subfigure[Feasibility ($\log$)]{
\includegraphics[width=5cm]{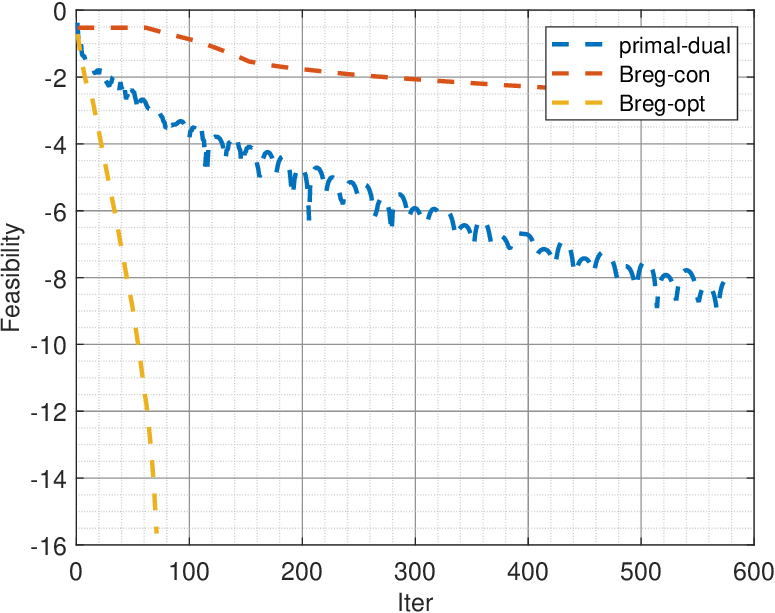}}
\caption{Comparison for the $\ell_\infty$ regularization problem}
\label{least-linf}
\end{figure}

\section{Conclusion}
This paper introduces a novel bilevel optimization formulation, offering a cut-and-project perspective to reevaluate Bregman regularized iterations. This approach not only encompasses existing methods but also extends to tackle broader inverse problems, including certain sparse noise models. Through detailed analysis, we explore the impact of step-sizes on algorithmic performance, delineating a practical range for their selection and providing a convergence guarantee to a feasible point. Our unified convergence condition offers a less restrictive criterion compared to existing literature, enhancing the applicability of our framework. Additionally, we introduce a Bregman distance growth condition, proving the linear convergence of our algorithm.

\bibliography{ref}

\bibliographystyle{plain}
\end{document}